\documentclass{amsart}
\usepackage{amsmath,amsfonts,amssymb,amsthm,amscd}
\usepackage{color}

\numberwithin{equation}{section}

\newcommand{\CP}{{\mathbb C}\!\operatorname{P}^1}

\newcommand{\SL}{\operatorname{SL}(2,{\mathbb R})}
\newcommand{\SLR}{\operatorname{SL}(2,\mathbb R)}
\newcommand{\GLR}{\operatorname{GL}(2,\mathbb R)}
\newcommand{\PSL}{\operatorname{PSL}(2,{\mathbb R})}
\newcommand{\PGL}{\operatorname{PGL}(2,{\mathbb R})}

\DeclareMathOperator{\diam}{\operatorname{diam}}

\newcommand{\cB}{\mathcal{B}}
\newcommand{\cC}{\mathcal{C}}

\newcommand{\cH}{\mathcal{H}}
\newcommand{\cL}{\mathcal{L}}
\newcommand{\cM}{\mathcal{M}}
\newcommand{\cQ}{\mathcal{Q}}

\newcommand\Z{\mathbb Z}
\newcommand\ZZ{\mathbb Z}
\newcommand\N{\mathbb N}

\newcommand\R{\mathbb R}
\newcommand\C{\mathbb C}

\newcommand\Dk{\Delta\widetilde{\kappa}}
\newcommand\Dc{\Delta\widetilde{c}_{\mathit{area}}}
\newcommand{\carea}{c_{\mathit{area}}}
\newcommand{\noz}{n}

\renewcommand{\epsilon}{\varepsilon}
\newlength{\halfbls}

\newtheorem{Theorem}{Theorem}
\newtheorem*{NNTheorem}{Theorem}

\newtheorem{Proposition}{Proposition}

\newtheorem{Lemma}{Lemma}[section]

\newtheorem{Corollary}{Corollary}

\newtheorem{Problem}{Problem}
\newtheorem{Question}{Question}

\theoremstyle{remark}

\newtheorem{Remark}{Remark}
\begin{document}

\begin{picture}(0,0)(0,-15)
%\put(-13,40){
\put(15,40){
\parbox{300pt}{\small\textit{
\hspace*{-5pt}
``You, my forest and water! One swerves, while the other shall spout\\
\hspace*{5pt} Through your body like draught; one declares, while the first has a doubt.''\\
\hspace*{180pt}J.~Brodsky}}
}
\end{picture}

\begin{picture}(0,0)(0,20)
%\put(-13,40){
\put(150,40){
\parbox{7cm}{\small\textit{
\hspace*{-5pt}``Mein Vater, mein Vater, und hoerest du nicht,\\
Was Erlenkoenig mir leise verspricht?''\\
``Sei ruhig, bleib ruhig, mein Kind!\\
In duerren Blaettern saeuselt der Wind.''\\
\hspace*{3.5cm}J.~W.~Goethe}}
}
\end{picture}

\title{Cries and Whispers in Wind-tree Forests}

\date{February 21, 2014}

\dedicatory{
To  the  memory  of Bill Thurston\\
with admiration for his fantastic imagination.}

% First author
\author{Vincent Delecroix}
\thanks{Research  of  the  authors  is  partially  supported  by  ANR
``GeoDyM''}
\address{
LaBRI,
Domaine universitaire,
351 cours de la Lib\'eration, 33405 Talence, FRANCE
}
\email{20100.delecroix@gmail.com}

% Second author
\author{Anton Zorich}
\thanks{Research of the second author is partially supported by IUF}
\address{
Institut Universitaire de France;
Institut de Math\'matiques de Jussieu --
Paris Rive Gauche,
UMR7586,
B\^atiment Sophie Germain,
Case 7012,
75205 PARIS Cedex 13, France}
\email{anton.zorich@imj-prg.fr}

\begin{abstract}
We study billiard in the plane endowed with symmetric $\Z^2$-periodic
obstacles  of  a  right-angled  polygonal  shape.  One  of  our  mane
interests  is dependence of the diffusion rate of the billiard on the
shape  of the obstacle. We prove, in particular, that when the number
of angles of a symmetric connected obstacle grows, the diffusion rate
tends to zero, thus answering a question of J.-C.~Yoccoz.

Our  results  are  based  on computation of Lyapunov exponents of the
Hodge  bundle  over  hyperelliptic  loci  in  the  moduli  spaces  of
quadratic  differentials,  which  represents independent interest. In
particular,  we  compute  the  exact  value  of the Lyapunov exponent
$\lambda^+_1$  for  all elliptic loci of quadratic differentials with
simple zeroes and poles.
\end{abstract}

\maketitle
%\tableofcontents

%-----------------------------------------------------------------
\section{Introduction}

The  classical wind-tree model corresponds to a billiard in the plane
endowed  with  $\Z^2$-periodic  obstacles  of  rectangular shape; the
sides   of   the  rectangles  are  aligned  along  the  lattice,
see Figure~\ref{fig:windtree}.

\includegraphics{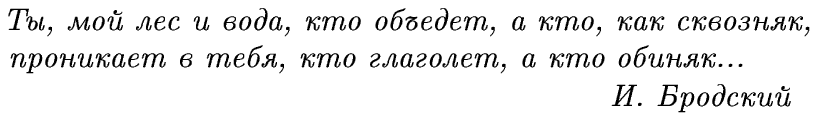}

\begin{figure}[htb]
%
%  Wind-tree
   %
\includegraphics{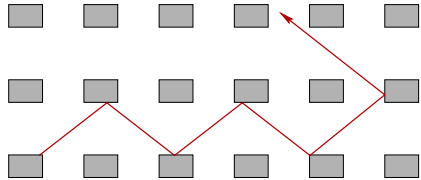}
\vspace{70bp}
\caption{
\label{fig:windtree}
Original wind-tree model.
   }
\end{figure}

The    wind-tree   model
(in a slightly different version)
was   introduced   by   P.~Ehrenfest   and
T.~Ehrenfest~\cite{Ehrenfest}  about  a  century  ago and studied, in
particular,  by  J.  Hardy and J.~Weber~\cite{Hardy:Weber}. All these
studies had physical motivations.

Several advances were obtained recently using the powerful technology
of  deviation  spectrum  of  measured  foliations on surfaces and the
underlying  dynamics  in  the moduli space. For all parameters of the
obstacle  and for almost all directions the trajectories are known to
be recurrent~\cite{Avila:Hubert}; there are examples of
divergent                   trajectories                  constructed
in~\cite{Delecroix:divergent:trajectories};   the  non-ergodicity  is
proved       in~\cite{Fraczek:Ulcigrai}.      It      was      proved
in~\cite{Delecroix:Hubert:Lelievre}   that   the  diffusion  rate  is
$\frac{2}{3}$;
it does not depend either on the concrete values of
parameters of the obstacle or on almost any direction and almost any
starting point, see Figure~\ref{fig:windtree:variations}.

\begin{figure}[htb]
%
%  Wind-trees of diffferent sizes
%
\includegraphics{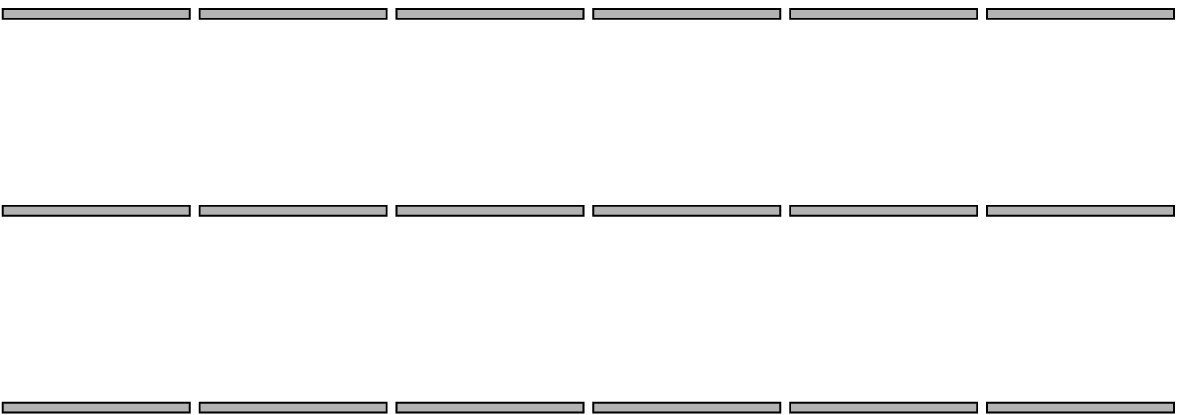}
\includegraphics{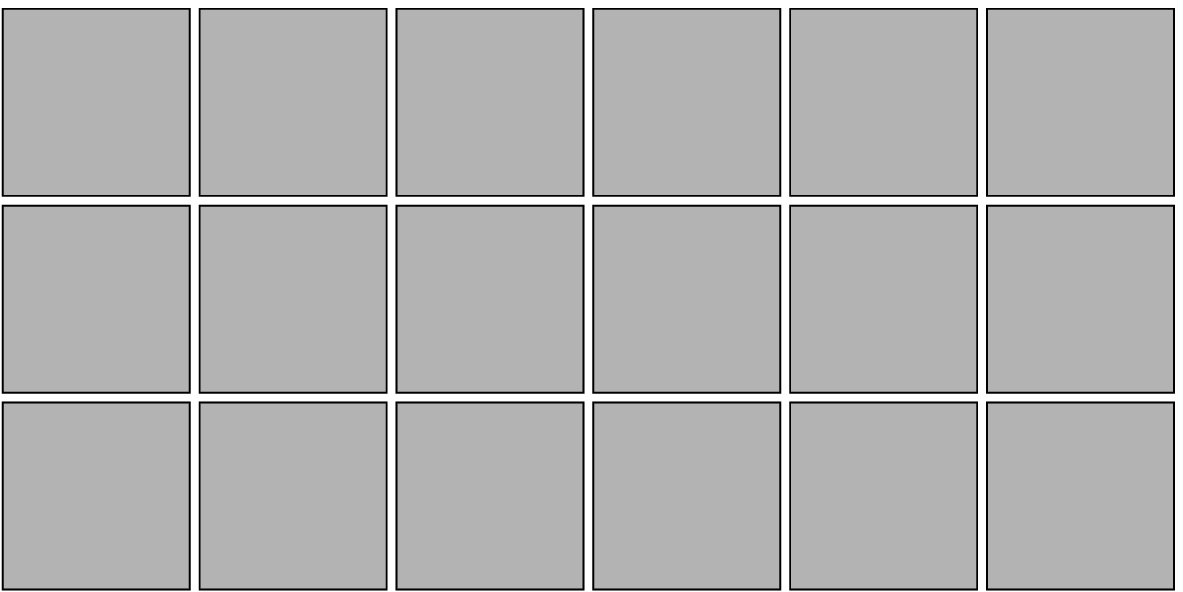}
\vspace{65bp}
\caption{
\label{fig:windtree:variations}
The diffusion rate $\frac{2}{3}$ does not depend on particular values
of  the  parameters  of the rectangular scatterer: it is the same for
the  plane  with  horizontal  walls  having  tiny periodic holes, for
narrow  periodic  corridors  between ``chocolate plates'' and for any
other periodic billiard as in Figure~\ref{fig:windtree}.
   }
\end{figure}

In  other  words,  the  maximal  deviation of the trajectory from the
starting point during the time $t$ has the order of $t^{\frac{2}{3}}$
for  large  $t$ in the following sense:
$$
\lim_{t\to\infty}
\frac{\log\diam(\text{trajectory for time interval}[0,t])}{\log t}
=\frac{2}{3}\,.
$$
Thus,  this  behavior  is  quite different from the
brownian  motion, random walk in the plane, or billiards in the plane
with  periodic  dispersing  scatterers: for all of them the diffusion
has   the   order  $\sqrt{t}$  (and,  thus,  the  diffusion  rate  is
$\frac{1}{2}$).

We address the natural question ``what happens if we change the shape
of  the obstacle?''. We do not have ambition to solve this problem in
the  current  paper  in  the  most general setting. We just plant the
wind-tree  forest  with several interesting families of obstacles and
study  the diffusion rate as the combinatorics of the obstacle inside
the  family becomes more complicated. We show, in particular, that if
the  obstacle  is  a  connected  symmetric right-angled polygon as on
Figure~\ref{fig:windtree:swiss:cross}, then the diffusion rate in the
corresponding  wind-tree  model tends to zero as the number of corners
of  the obstacle grows; see Theorem~\ref{th:single:obstacle} for more
precise  statement.  This result gives an explicit affirmative answer
to a question addressed by J.-C.~Yoccoz.

\begin{figure}[htb]
%
%  Wind-tree Swiss cross
   %
\includegraphics{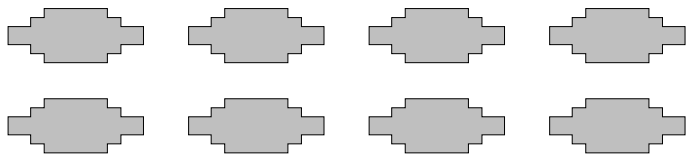}
\vspace{60bp}
\caption{
\label{fig:windtree:swiss:cross}
The  diffusion  rate  in  this wind-tree forest tends to zero when the
number of corners of the obstacle grows.
   }
\end{figure}

Now,  when  we have showed that for certain species of wind-trees the
sound  in  the  wind-tree  forest propagates ``as a whisper'' (in the
sense  that  the  diffusion  rate tends to zero), the challenge is to
prove  that  for certain other species it propagates ``as a cry'' with
the diffusion rate approaching $1$.

\begin{Question}
Are  there periodic wind-tree billiards with diffusion rate arbitrary
close to 1? Are there continuous families of wind-tree billiards like
this?  What  are  the shapes of the obstacles which provide diffusion
rate arbitrary close to 1?
\end{Question}

%--------------------------------------------------------------------
\subsection{Strategy of the proof.}
\label{ss:strategy}
We  develop  the  approach  originated  in  the  pioneering  work  of
\mbox{V.~Delecroix},               \mbox{P.~Hubert},              and
S.~Leli\`evre~\cite{Delecroix:Hubert:Lelievre}    who   applied   the
results  from  dynamics  in  the  moduli space to the wind-tree model.
Several very deep recent advances in dynamics in moduli spaces are of
crucial importance for us.

Following~\cite{Delecroix:Hubert:Lelievre}    we    reformulate   the
original  billiard problem in terms of the deviation spectrum for the
leaves  of  directional  measured  foliations  on the associated flat
surface $S$. This part is quite elementary and straightforward.

By  recent  deep result of J.~Chaika and A.~Eskin~\cite{Chaika:Eskin}
(based   on   fundamental   advances   of   A.~Eskin,  M.~Mirzakhani,
A.~Mokhammadi~\cite{Eskin:Mirzakhani},
\cite{Eskin:Mirzakhani:Mohammadi})    almost    all   directions   on
\textit{every}  flat  surface  are  Lyapunov-generic.  Combining  the
techniques~\cite{Forni},  \cite{Zorich:How:do} of deviation
spectrum  with  the  results from~\cite{Delecroix:Hubert:Lelievre} we
conclude  that  the  diffusion  rate  of the wind-tree billiard as in
Figure~\ref{fig:windtree:swiss:cross} equals to the Lyapunov exponent
$\lambda^+(h^\ast)=\lambda^+(v^\ast)$   of   certain   very  specific
integer cocycles $h^\ast,v^\ast\in H^1(S,\Z)$, where the flat surface
$S$  is  considered  as  a point of the orbit closure $\cL(S)$ in the
ambient  stratum  of  meromorphic quadratic differentials; the vector
space $H^1(S,\C)$ containing $h^\ast$ and $v^\ast$ is considered as
a  fiber  over $S$ of the complex Hodge bundle over $\cL(S)$, and the
Lyapunov  exponents  are  the Lyapunov exponents of the complex Hodge
bundle  $H^1_{\C}$  over $\cL(S)$ with respect to the Teichm\"uller
geodesic flow.

In  the  current  paper  we  intentionally  focus  on  the  family of
billiards  as in Figure~\ref{fig:windtree:swiss:cross} since for this
family  the  rest  of  the  computation  is particularly transparent.
Namely,  the  flat surface $S$ belongs to the \textit{hyperelliptic} locus
$\cL:=\cQ^{hyp}(1^{2m},-1^{2m})$  over $\cQ(1^m,-1^{m+4})$.
(In our particular situation, the \textit{hyper}elliptic
locus is, actually, ``elliptic'': the genus of the covering surface is $1$.)
Moreover,
applying   the   arguments  analogous  to  those  in~\cite{AEZ},  one
immediately  verifies  that the family $\cB$ of billiards is so large
(in  dimension)  that  it is transversal to the unstable foliation in
the  ambient invariant submanifold $\cL$, and, thus, for almost every
billiard  table  $\Pi$  in  $\cB$  the  orbit closure $\cL(S)$ of the
associated  flat  surface  $S(\Pi)$  coincides  with the entire locus
$\cL$.

The  flat  surface  $S$  has  genus  one, so the complex Hodge bundle
$H^1_{\C}=H^1_+$    has    single    positive   Lyapunov   exponent
$\lambda_1^+$.  The  fact that $h,v$ are \textit{integer} immediately
implies that $\lambda^+(h)=\lambda^+(v)=\lambda^+_1$.

Formula~(2.3)   from~\cite{EKZ}   expresses   the  sum  $\sum_{i=1}^g
\lambda^+_i$  of  all positive Lyapunov exponents of $H^1_+$ in terms
of  the  degrees  of  zeroes  (and poles) in the ambient locus and in
terms  of  the  Siegel--Veech  constant  $\carea(\cL)$.  Since in our
particular  case  the  genus of the surface is equal to one, we get a
formula  for  the individual Lyapunov exponent $\lambda_1^+$ in which
we are interested.

Developing  Lemma~(1.1)  from~\cite{EKZ}  we relate the Siegel--Veech
constant     $\carea(\cL)$     of     the     hyperelliptic     locus
$\cL:=\cQ^{hyp}(1^{2m},-1^{2m})$   over  $\cQ(1^m,-1^{m=4})$  to  the
Siegel--Veech  constants $c_\cC(\cQ(1^m,-1^{m+4}))$ of the underlying
stratum  in  genus  zero.  Plugging  in  the resulting expression the
explicit  values of $c_\cC(\cQ(1^m,-1^{m+4}))$ obtained in the recent
paper~\cite{AEZ}  and  proving certain combinatorial identity for the
resulting  hypergeometric sum we obtain the desired explicit value of
$\lambda^+_1(\cQ^{hyp}(1^{2m},-1^{2m}))$,     which represents     the
diffusion rate in almost every original billiard.

\begin{Remark}
Our results provide certain evidence that when the genus is fixed and
the number of simple poles grows, the Lyapunov exponents of the Hodge
bundle  tend  to  zero  (see~\cite{Grivaux:Hubert}  for  the original
conjecture).
\end{Remark}

\subsection{Structure of the paper}
In  section~\ref{s:Main:Results}  we  state  the  main  result in two
different  forms.  In section~\ref{s:From:billiards:to:flat:surfaces}
we  show  how  to  reduce  the  problem  of  the  diffusion rate in a
generalized  wind-tree  billiard  to the problem of evaluation of the
top  Lyapunov exponent $\lambda_1^+$ of the complex Hodge bundle over
an   appropriate   hyperelliptic  locus  of  quadratic  differentials.  In
section~\ref{ss:Original:wind:tree:revisited} we revisit the original
paper~\cite{Delecroix:Hubert:Lelievre} where this question is treated
in  all  details  for  the  original  wind-tree  model  with periodic
rectangular scatterers. We suggest, however, several simplifications.
Namely,  in  section~\ref{ss:elliptic:locus} we describe the hyperelliptic
locus  over certain stratum of meromorphic quadratic differentials in
genus  zero  where  lives  the  flat surface $S$ corresponding to the
wind-tree billiard and we show that the diffusion rate corresponds to
the  top  Lyapunov exponent $\lambda_1^+$ of the complex Hodge bundle
over   the  $\PSL$-orbit  closure  of  $S$.  Following  an  analogous
statement  in~\cite{AEZ}  we prove in section~\ref{ss:Orbit:closures}
that  for  almost  any  initial billiard table $\Pi$ the $\PSL$-orbit
closure  of  the  associated flat surface $S(\Pi)$ coincides with the
entire  hyperelliptic  locus.  At  this  stage  we  reduce  the problem of
evaluation of the diffusion rate for almost all billiard table in the
family  to  evaluation  of  the  single  positive  Lyapunov  exponent
exponent  of the complex Hodge bundle $H^1_+=H^1_{\C}$ over certain
specific hyperelliptic locus.

In   section~\ref{s:Siegel:Veech:constant:for:the:elliptic:locus}  we
evaluate   this   Lyapunov   exponent.   We  start  by  recalling  in
section~\ref{ss:sum:of:lambda:plus} the technique from~\cite{EKZ}; we
also relate the Siegel--Veech constant of the hyperelliptic locus with the
Siegel--Veech  constant of the corresponding stratum in genus $0$. In
section~\ref{ss:genus:zero} we summarize the necessary material on
cyclinder   configurations   in   genus   zero  and  on  the  related
Siegel--Veech       constants       from~\cite{Boissy:configurations}
and~\cite{AEZ}. Finally, in section~\ref{ss:one:complicated:obstacle}
we                   prove                   the                  key
Theorem~\ref{th:lambda1:as:ratio:of:double:factorials} evaluating the
desired   Lyapunov   exponent   $\lambda_1^+$.   The   proof  uses  a
combinatorial  identity for certain hypergeometric sum; this identity
is proved separately in section~\ref{s:combinatorial:identities}.

Following  the  title ``\textit{What's next?}'' of the conference, we
discuss   in   appendix~\ref{a:What:is:next}  directions  of  further
research in the area relevant to the context of this paper.

% #############################################################
% #############################################################
% #############################################################

\section{Main results}
\label{s:Main:Results}

Denote by $\cB(m)$ the family of billiards such that the obstacle has
$4m$  corners  with  the  angle  $\pi/2$. Say, all billiards from the
original    wind-tree    family   as   in   Figures~\ref{fig:windtree}
and~\ref{fig:windtree:variations}  live  in $\cB(1)$; the billiard in
Figure~\ref{fig:windtree:swiss:cross}   belongs   to   $\cB(3)$;  the
billiard    in   Figure~\ref{fig:windtree:swiss:cross}   belongs   to
$\cB(17)$.

\begin{Theorem}
\label{th:single:obstacle}
For almost all billiard tables in the family $\cB(m)$
and for almost all directions
the diffusion rate $\delta(m)$ is the same and equals
$$
\delta(m)=\frac{(2m)!!}{(2m+1)!!}\,.
$$
When $m\to+\infty$ $\delta(m)$ has asymptotics
$$
\delta(m)=
\frac{\sqrt{\pi}}{2\sqrt{m}}
\left(1+O\left(\frac{1}{m}\right)\right)\,.
$$

\end{Theorem}

Here   the   double   factorial   means   the  product  of  all  even
(correspondingly    odd)   natural   numbers   from   $2$   to   $2m$
(correspondingly from $1$ to $2m+1$). For the original wind-tree, when
the  obstacle  is  a  rectangle,  we  have $m=1$ and we get the value
$\delta(1)=\frac{2}{3}$ found in~\cite{Delecroix:Hubert:Lelievre}.

\begin{figure}[htb]
%
%  Wind-tree Swiss cross
   %
\includegraphics{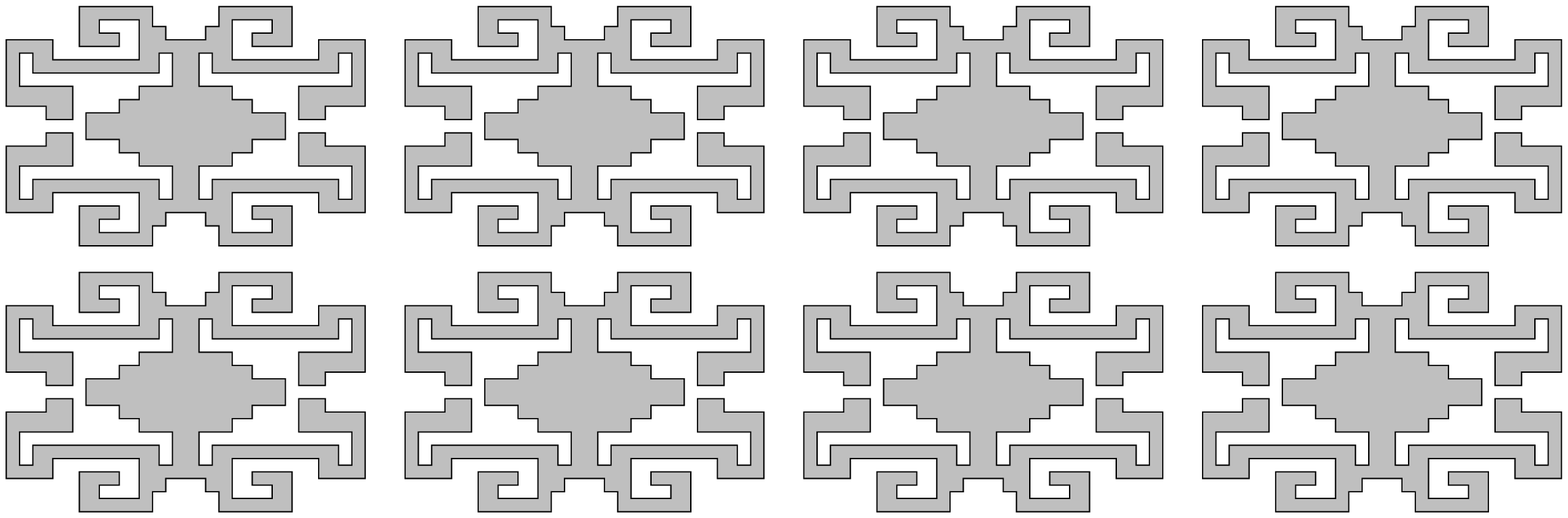}
\vspace{85bp}
\caption{
\label{fig:windtree:swiss:cross:alternative}
The  diffusion  rate  depends  only  on  the number of corners of the
obstacle  and  not  on  the  particular values of (almost all) length
parameters nor on the particular shape of the obstacle.
   }
\end{figure}

Following  the  strategy  described  in  section~\ref{ss:strategy} we
derive Theorem~\ref{th:single:obstacle} from the following result.

\begin{Theorem}
\label{th:lambda1:as:ratio:of:double:factorials}
The  locus  $\cQ^{hyp}(1^{2m},-1^{2m})$ over $\cQ(1^{m},-1^{m+4})$ is
connected  and  invariant  under  the  action  of $\PGL$. The measure
induced   on   $\cQ_1^{hyp}(1^{2m},-1^{2m})$  from  the  Masur--Veech
measure  on  $\cQ_1(1^m,-1^{m+4})$  is  $\PSL$-invariant  and ergodic
under  the  action  of  the Teichm\"uller geodesic flow. The Lyapunov
exponent  $\lambda^+_1(m)$ of the Hodge bundle $H^1_+$ over the locus
$\cQ^{hyp}(1^{2m},-1^{2m})$  under  consideration  has  the following
value:
\begin{equation}
\label{eq:lambda:plus:m}
\lambda^+_1(m)=\frac{(2m)!!}{(2m+1)!!}\,.
\end{equation}
\end{Theorem}

\begin{Remark}
Note  that  there  is a very important difference between the case of
$m=1$  (corresponding to the classical wind-tree) and the cases $m\ge
2$. Namely, the locus $\cQ^{hyp}(1^{2},-1^{2})$ over $\cQ(1,-1^5)$ is
\textit{nonvarying}:  the Lyapunov exponent $\lambda_1^+=\frac{2}{3}$
for      \textit{all}      flat     surfaces     in     the     locus
$\cQ^{hyp}(1^{2},-1^{2})$.
For  $m\ge  2$  it  is  not  true  anymore.
First of all, for each integer $m \geq 2$, taking
appropriate unramified covering of
degree $m$ of a flat surface in the stratum $\cQ(1^2,-1^2)$
we get a flat surface in the hyperelliptic locus of $\cQ^{hyp}(1^{2m},-1^{2m})$ over $\cQ(1^m, -1^{m+4})$.
Concretely, such surface can be built starting from the original wind-tree
model with rectangles and taking a fundamental domain that is made
of $m$ copies of the unit square.

By construction, the Lyapunov exponent $\lambda_1^+$
of the resulting Teichm\"uller curve in $\cQ^{hyp}(1^{2m},-1^{2m})$,
and, hence, the diffusion rate $\delta$ for the
corresponding wind-tree billiard does not change: $\delta=\lambda_1^+=2/3$.
Furthermore, for $m=2$ we were able to find examples of square tiled surfaces for which
the value is neither the generic value $\delta(2)=8/15=0.5333\ldots$ nor $2/3=0.6666\ldots$
\[
\begin{array}{|l|l|}
\hline
\text{permutations $r$ and $u$} & \lambda^+_1 \\
\hline
\begin{array}{l}
r=(1,2,3,4,5,6,7)(8,9,10,11,12,13,14) \\
u=(1,3,13,8,2,14)(4,6,11,5,10,12)(7,9)
\end{array}
& \frac{20}{33} = 0.6060\ldots \\ \hline
\begin{array}{l}
r=(1,2,3,4,5,6,7,8)(9,10,11,12,13,14) \\
u=(1,2,3,14,9)(4,13)(5,6,7,11,12)(8,10)
\end{array}
& \frac{6}{11} = 0.5454\ldots \\ \hline
\end{array}
\]
See also the example in Appendix~\ref{a:removing:squares:in:the:wind:tree}
with $m=3$.
\end{Remark}
\begin{Question}
What  are  the  extremal  values  of  $\lambda^+_1$  over  all closed
$\PSL$-invariant     suborbifolds     in
a given stratum (given locus)? Same question
for the concrete hyperelliptic    locus
$\cQ^{hyp}(1^{2m},-1^{2m})$  over $\cQ(1^{m},-1^{m+4})$? What are the
shapes of billiards for which these values are achieved (if there are
any wind-tree billiards corresponding to these invariant suborbifolds)?
\end{Question}

% #############################################################
% #############################################################
% #############################################################

\section{From billiards to flat surfaces}
\label{s:From:billiards:to:flat:surfaces}

\subsection{Original wind-tree revisited}
\label{ss:Original:wind:tree:revisited}
Recall that in the classical case of a billiard in a rectangle we can
glue a flat torus out of four copies of the billiard table and unwind
billiard trajectories to flat geodesics on the resulting flat torus.

In  the  case of the wind-tree model we also start from gluing a flat
surface  out  of  four  copies  of  the billiard table. The resulting
surface  is $\Z\!\oplus\!\Z$-periodic with respect to translations by
vectors  of  the  original  lattice.  We  pass  to  the quotient over
$\Z\!\oplus\!\Z$  to get a compact flat surface without boundary. For
the  case  of  the  original  wind-tree  billiard  the resulting flat
surface           $X$           is           represented           at
Figure~\ref{fig:windtree:surface:genus:5}.   It  has  genus  $5$;  it
belongs    to    the    stratum    $\cH(2^4)$    (see    section    3
of~\cite{Delecroix:Hubert:Lelievre} for details).

One of the key statements of~\cite{Delecroix:Hubert:Lelievre} can be
stated as follows.

Let  $\Pi$  be  the original rectangular obstacle, define the corresponding wind-tree
billiard by the same symbol $\Pi$.
Let $X=X(\Pi)$ be the flat
surface  as  in Figure~\ref{fig:windtree:surface:genus:5} constructed
by the wind-tree billiard defined by the obstacle $\Pi$.

Consider the $\SLR$-orbit closure $\cL(X)\subset\cH(2^4)$ of the flat
surface  $X$.  Consider  the  cohomology  classes  $h^\ast, v^\ast\in
H^1(X,\Z)$ Poincar\'e-dual to cycles
\begin{align*}
h&=h_{00}-h_{01}+h_{10}-h_{11}\\
v&=v_{00}-v_{10}+v_{01}-v_{11}
\end{align*}
(see  Figure~\ref{fig:windtree:surface:genus:5})  as  elements of the
fiber  over  the  point  $X\in\cL(X)$  of  the  complex  Hodge bundle
$H^1_{\C}$  over  $\cL(X)$.

\begin{NNTheorem}[\cite{Delecroix:Hubert:Lelievre}]
The diffusion rate in the original wind-tree billiard $\Pi$ coincides
with  the Lyapunov exponent $\lambda(h^\ast)=\lambda(v^\ast)$) of the
complex  Hodge  bundle $H^1_{\C}$ with respect to the Teichm\"uller
geodesic flow on $\cL(X)$.
\end{NNTheorem}

\begin{figure}[htb]
%\centering
%
%  Unfolded wind-tree
   %
\includegraphics{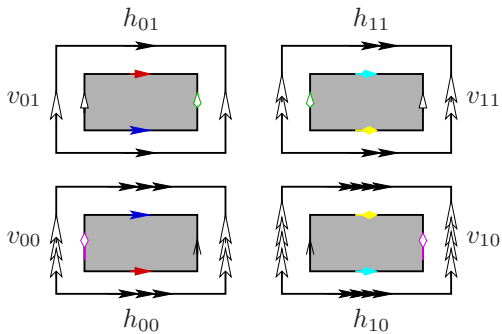}
\begin{picture}(0,0)(28,11)
\put(-32,-110){$h_{00}$}
\put(-32,4){$h_{01}$}
\put(55,-110){$h_{10}$}
\put(55,4){$h_{11}$}
\put(-76,-78){$v_{00}$}
\put(98,-78){$v_{10}$}
\put(-76,-25){$v_{01}$}
\put(98,-25){$v_{11}$}
\end{picture}
\vspace{120bp}
\caption{
\label{fig:windtree:surface:genus:5}
The  flat  surface $X$ obtained as a quotient over $\Z\oplus\Z$ of an
unfolded wind-tree billiard table.
   }
\end{figure}

\begin{Remark}
\label{rm:Chaika:Eskin}
Recent  result  in~\cite{Chaika:Eskin}  proving that for \textit{any}
flat  surface almost all directions on it are Lyapunov-generic allows
to  simplify  part of the argument in the proof of the above Theorem.
In    particular,    it    justifies   that   $\lambda(h^\ast)$   and
$\lambda(v^\ast)$  are  well-defined  for $X$ endowed with almost all
direction.
\end{Remark}

Note    that    any    resulting    flat    surface    $X$    as   in
Figure~\ref{fig:windtree:surface:genus:5}  has  (at  least) the group
$(\Z/2\Z)^3$  as  a  group  of isometries. As three generators we can
choose  the  isometries $\tau_h$ and $\tau_v$ interchanging the pairs
of flat tori with holes in the same rows (correspondingly columns) by
parallel  translations and the isometry $\iota$ acting on each of the
four  tori  with holes as the central symmetry with the center in the
center of the hole.

\begin{figure}[htb]
\centering
%
%  One symmetric obstacle; no walls
   %
\includegraphics{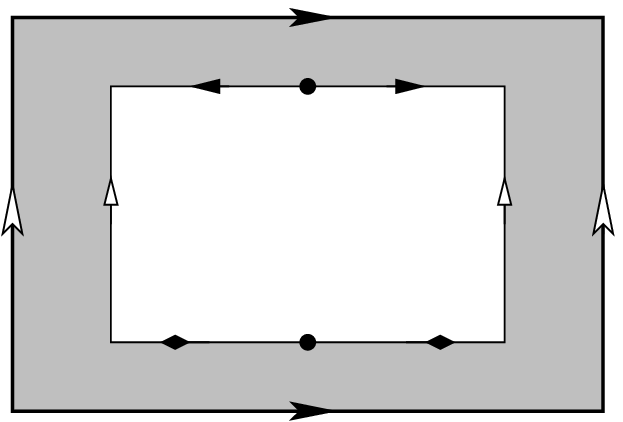}
\includegraphics{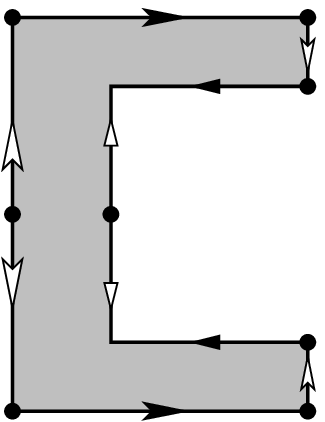}
\begin{picture}(0,0)(0,0)
\put(48,-64){$P_1$}
\put(48,-6){$P_2$}
\put(149,-6){$P_3$}
\put(149,-32){$P_4$}
\put(92,-32){$P_5$}
\put(92,-64){$P_6$}
\end{picture}
\vspace{120bp}
\caption{
\label{fig:windtree:eliptic:cover}
A  surface $\tilde S$ in the hyperelliptic locus $\cQ^{hyp}(1^2,-1^2)$ (on
the  left)  is  a  double  cover  over  the underlying surface $S$ in
$\cQ(1,-1^5)$   (on   the   right)  ramified  at  four  simple  poles
represented by bold dots.
   }
\end{figure}

Consider the quotient $\tilde S$ of the flat surface $\tilde{S}$ over
the  subgroup  $(\Z/2\Z)^2$  of  isometries  spanned  by $\tau_h$ and
$\iota\circ\tau_v$.  The  resulting surface $\tilde S$ as on the left
side  of  the  Figure~\ref{fig:windtree:eliptic:cover} belongs to the
stratum $\cQ(1^2,-1^2)$; in particular, it has genus $1$. The surface
$S$  obtained  as  the quotient of the original flat surface $X$ over
the   entire   group  $(\Z/2\Z)^3$  as  on  the  right  side  of  the
Figure~\ref{fig:windtree:eliptic:cover}   belongs   to   the  stratum
$\cQ(1,-1^5)$;  in  particular, it has genus $0$. Clearly, $\tilde S$
is  a ramified double cover over $S$ with ramification points at four
(out of five) simple poles of the flat surface $S$.

\begin{Lemma}
Consider  the  natural  projection  $p:X\to  \tilde S$. The following
inclusion is valid:
$$
h^\ast\in p^\ast(H^1(\tilde S;\Z{}))\,.
$$
\end{Lemma}
\begin{proof}
Recall that
$$
h=h_{00}-h_{01}+h_{10}-h_{11}\,,
$$
where    the    cycles   $h_{ij}$,   $i,j=0,1$   are   indicated   in
Figure~\ref{fig:windtree:surface:genus:5}. Now note that
$$
\tau_h(h_{00})=h_{10}\qquad    \iota\circ\tau_v(h_{00})=-h_{01}\qquad
\iota\circ\tau_v\circ\tau_h(h_{00})=-h_{11}\qquad\,.
$$
Thus,  the  cycle  $h$  is invariant under the action of the commutative
group  $\Gamma$  spanned  by  $\tau_h$  and  $\iota\circ\tau_v$. This
implies  that  the Poincar\'e-dual cocycle $h^\ast$ is also invariant
under the action of $\Gamma$ and, hence, is induced from some cocycle
in cohomology of $\tilde S=X/\Gamma$.
\end{proof}

\begin{Corollary}
\label{cor:lex:m:equals:1}
Consider  the  orbit  closure  $\cM(\tilde  S)$  of  the flat surface
$\tilde S(\Pi) \in \cQ(1^2,-1^2)$. The diffusion rate in the original
wind-tree  billiard  $\Pi$  in  vertical direction coincides with the
positive  Lyapunov exponent $\lambda_1^+$ of the complex Hodge bundle
$H^1_{\C}$  with  respect  to  the  Teichm\"uller  geodesic flow on
$\cM(\tilde S)$.
\end{Corollary}
\begin{proof}
From the Theorem of Delecroix--Hubert--Leli\`evre cited above we know
that   the  diffusion  rate  coincides  with  the  Lyapunov  exponent
$\lambda(h^\ast)$.  By  the  previous  Lemma  $h^\ast=p^\ast(\alpha)$
where  $\alpha\in  H^1(\tilde  S;\Z)$.  It  is  immediate to see that
$\lambda(h^\ast)=\lambda(\alpha)$, where $\lambda(\alpha)$ is already
the Lyapunov exponent of the of the complex Hodge bundle $H^1_{\C}$
with  respect  to the Teichm\"uller geodesic flow on $\cM(\tilde S)$.
Since   $g(\tilde   S)=1$  the  corresponding  cocycle  has  Lyapunov
exponents  $\pm\lambda_1^+$.  Since  $\alpha$ is an integer covector,
$\lambda(\alpha)$     cannot     be    strictly    negative.    Hence
$\lambda(\alpha)=\lambda^+_1$.
\end{proof}

%--------------------------------------------------------------------
\subsection{Elliptic locus and diffusion in the generalized
wind-tree billiard.}
\label{ss:elliptic:locus}
The hyperelliptic locus  $\cQ^{hyp}(1^{2m},-1^{2m})$   in   $\cQ(1^{2m},-1^{2m})$  is
obtained  by  the following construction. For any flat surface $S$ in
$\cQ(1^m,-1^{m+4})$  consider  all  possible  quadruples of unordered
simple  poles.  For  each quadruple construct a ramified double cover
with  four  ramification  points  exactly  at the chosen quadruple of
points.  By  construction  the  induced  flat  surface belongs to the
stratum  $\cQ(1^{2m},-1^{2m})$ in genus $1$. Considering all possible
flat  surfaces  $S$  in  $\cQ(1^m,-1^{m+4})$  and all covers over all
quadruples     of     simple     poles     we     get    the    locus
$\cQ^{hyp}(1^{2m},-1^{2m})$ over $\cQ(1^{m},-1^{m+4})$.

It  is  immediate  to see that when the obstacle is a rectangle as in
the original wind-tree billiard the surface $\tilde S$ constructed in
section~\ref{ss:Original:wind:tree:revisited} belongs to the hyperelliptic
locus   $\cQ^{hyp}(1^2,-1^2)$   over   the   stratum   $\cQ(1,-1^5)$.
Similarly, when the obstacle has $4m$ corners with the angle $\pi/2$,
the   analogous   surface   $\tilde   S$   (as   on   the   left   of
Figure~\ref{fig:swiss:cross})   belongs   to   the   hyperelliptic   locus
$\cQ^{hyp}(1^{2m},-1^{2m})$ over $\cQ(1^{m},-1^{m+4})$.

\begin{figure}[htb]
\centering
%
%  One symmetric obstacle; no walls
   %
\includegraphics{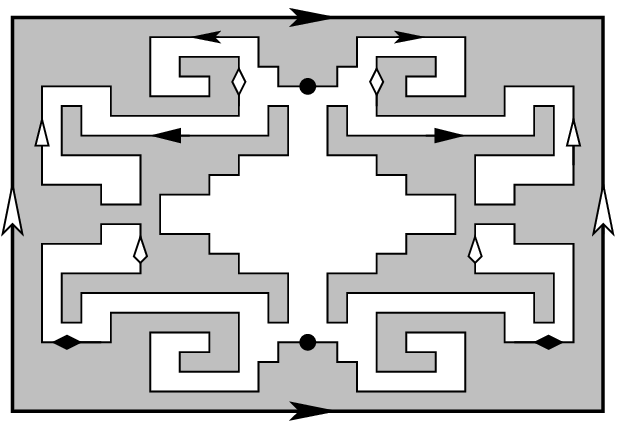}
\includegraphics{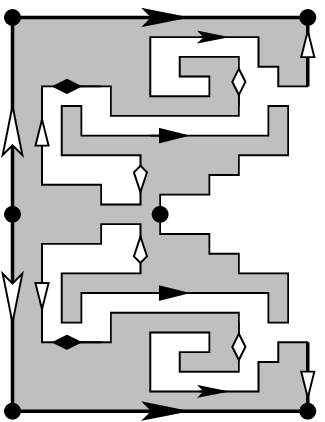}
\vspace{120bp}
\caption{
\label{fig:swiss:cross}
A  surface $\tilde S$ in the hyperelliptic locus $\cQ^{hyp}(1^{2m},-1^{2m})$
is   a   double   cover   over   the   underlying   surface   $S$  in
$\cQ(1^{m},-1^{m+4})$  branched  at the four simple poles represented
by bold dots.
   }
\end{figure}

The                            arguments                           of
Delecroix--Hubert--Leli\`evre~\cite{Delecroix:Hubert:Lelievre} extend
to the more general case of obstacles $\Pi$ symmetric with respect to
vertical  and  horizontal  axes  with  arbitrary  numbers  of  angles
$\frac{\pi}{2}$, basically, line by line with an extra simplification
due       to       results~\cite{Chaika:Eskin}      mentioned      in
Remark~\ref{rm:Chaika:Eskin}. Applying exactly the same consideration
as    above   we   prove   the   following   statement   generalizing
Corollary~\ref{cor:lex:m:equals:1} to arbitrary $m\in\N$.

Let  $\Pi$  be  a  connected  obstacle having $4m$ corners with angle
$\frac{\pi}{2}$  and  $4(m-1)$  corners  with  angle $\frac{3\pi}{2}$
aligned  in  such  way  that all its sides are vertical or horizontal
(see     the    white    domain    in    the    left    picture    in
Figure~\ref{fig:swiss:cross}).  Suppose that $\Pi$ is symmetric under
reflections over some vertical and some horizontal lines.

\begin{Proposition}
\label{pr:diffusion:as:Lyapunov:exponent}
Consider     the     orbit    closure    $\cM(\tilde    S)    \subset
\cQ^{hyp}(1^{2m},-1^{2m})$ of the flat surface $\tilde S(\Pi)$ in the
ambient  hyperelliptic  locus  $\cQ^{hyp}(1^{2m},-1^{2m})$.  The diffusion
rate  in the wind-tree billiard with periodic obstacles $\Pi$ aligned
with  the  lattice  coincides  with  the  positive  Lyapunov exponent
$\lambda_1^+$   of   the   complex  Hodge  bundle  $H^1_{\C}$  over
$\cM(\tilde  S)$  with  respect to the Teichm\"uller geodesic flow on
$\cM(\tilde S)$.
\end{Proposition}

%--------------------------------------------------------------------
\subsection{Orbit closures}
\label{ss:Orbit:closures}

Consider  the original wind-tree billiard with rectangular obstacles.
Such  a billiard is described by five real parameters: by two lengths
of  the  sides of the external rectangle defining the lattice; by two
lengths  of  the  sides  of  the  inner  rectangle represented by the
obstacle  and  by  the  angle  defining  the direction of trajectory,
see Figure~\ref{fig:windtree:eliptic:cover}.
Varying  continuously  these parameters we obtain a continuous family
of billiards.

Starting     with     a     more     general     obstacle    as    in
Figure~\ref{fig:swiss:cross}    having    $4m$   corners   of   angle
$\frac{\pi}{2}$,  $4(m-1)$  corners  of  angle  $\frac{3\pi}{2}$  and
symmetric with respect to vertical and horizontal axes of symmetry we
get  an  analogous  continuous family of billiards which we denote by
$\cB(m)$. It is immediate to check that
\begin{equation}
\label{eq:dim:cB}
\dim_{\R}\cB(m)=2m+2
\end{equation}
Here the  direction  of  the  billiard  flow is not considered
as a parameter of the family $\cB$. For the original
wind-tree    billiard   with   rectangular   obstacles   this   gives
$\dim_{\R}\left(\cB(1)\times \mathbb{S}^1/(\Z/2\Z)^2\right)=4+1=5$, as we have already seen.

\begin{Proposition}
\label{pr:almost:any:closure}
For  Lebesgue-almost every directional billiard in $\Pi\in\cB(m)$ the
$\GLR$-orbit  closure  of  $S(\Pi)$ in $\cQ(1^m,-1^{m+4})$ coincides
with the entire ambient stratum $\cQ(1^m,-1^{m+4})$.
\end{Proposition}
\begin{proof}
The     Proposition     is    a    straightforward    corollary    of
Lemma~\ref{lm:transversality} below.
\end{proof}

The  following  Lemma  is  completely  analogous  to  Proposition~3.2
in~\cite{AEZ}.

\begin{Lemma}
\label{lm:transversality}
Consider        the        canonical       local        embedding
$$
\cB(m)\times \left(\mathbb{S}^1/(\Z/2\Z)^2\right)\hookrightarrow\cQ(1^m,-1^{m+4})\,.
$$
For    almost    all    pairs    $(\Pi,\theta)$    in   $\cB(m)\times
\left(\mathbb{S}^1/(\Z/2\Z)^2\right)$  the  projection of the tangent
space                                        $T_\ast\big(\cB(m)\times
\left(\mathbb{S}^1/(\Z/2\Z)^2\right)\big)$  to  the unstable subspace
of the Teichm\"uller geodesic flow is a surjective map.
\end{Lemma}
\begin{proof}
Let  us  first  prove  the  statement  for  $m=1$  and  then make the
necessary  adjustments  for  the most general case. Consider a broken
line  following the upper part of the polygon as on the right side of
figure~\ref{fig:windtree:eliptic:cover}  starting  at the point $P_0$
and  finishing  at  the  corner  $P_4$.  Turn the figure by the angle
$\theta$.  Consider  the  associated  flat surface $S\in\cQ(1,-1^5)$.
Consider  the canonical orienting double cover $\hat S\in\cH(2)$. The
six  Weierstrass points of the cover correspond to the corners of our
polygonal pattern and to the two bold points on the horizontal axe of
symmetry. Thus the vectors of the resulting broken line considered as
complex numbers are exactly the basic half-periods of the holomorphic
1-form  $\omega$  corresponding  to a standard basis of cycles on the
hyperelliptic    surface    $\hat(S)$    (see    sections   3.1   and
3.3~in~\cite{AEZ}  for  more  details). The first cohomology $H^1(\hat S;\C)$
serve  as  local  coordinates in $\cQ(1,-1^5)$. The components of the
projection  of  the vector of periods to the space $H^1(\hat S;\R{})$
are of the form
$$
\pm 2\sin(\phi)|P_i P_{i+1}|\quad\text{or}\quad
\pm 2\cos(\phi)|P_i P_{i+1}|\quad\text{ where } i=1,\dots,4\,.
$$
Thus,  for  $\phi$  different from an integer multiple of $\pi/2$ the
composition map
$$
T_\ast\big(\cB(m)\times
\left(\mathbb{S}^1/(\Z/2\Z)^2\right)\big)\to
H^1(\hat S;\R{})
$$
is a surjective map.

It  is clear from the proof that to extend it from $m=1$ to arbitrary
$m\in\N$ it is sufficient to show that the real dimension of $\cB(m)$
coincides   with   the   complex   dimension  of  $\cQ(1^m,-1^{m+4})$.
Recalling~\eqref{eq:dim:cB}   and   the  classical  formula  for  the
dimension    of   a   stratum   $\cQ(d_1,\dots,d_n)$   of   quadratic
differentials
$$
\dim_{\C}\cQ(d_1,\dots,d_n)=2g+n-2
$$
we conclude that
$$
\dim_{\R}\cB(m)=2m+2=
\dim_{\C}\cQ(1^m,-1^{m+4})
$$
which completes the proof of the Lemma.
\end{proof}

Combining        Propositions~\ref{pr:diffusion:as:Lyapunov:exponent}
and~\ref{pr:almost:any:closure}      with      the      result     of
Chaika--Eskin~\cite{Chaika:Eskin}  telling  that for any flat surface
almost  all  directions are Lyapunov-generic, we obtain the following
Corollary,      which      proves      the      first     part     of
Theorem~\ref{th:single:obstacle}.

\begin{Corollary}
For  Lebesgue-almost every directional billiard in $\Pi\in\cB(m)$ the
diffusion       rate       in       almost       all       directions
$\theta\in\left(\mathbb{S}^1/(\Z/2\Z)^2\right)$   in   the  wind-tree
billiard   $\Pi$   coincides  with  the  positive  Lyapunov  exponent
$\lambda_1^+$  of  the  complex  Hodge  bundle  $H^1_{\C}$ over the
hyperelliptic  locus  $\cQ^{hyp}(1^{2m},-1^{2m})$ with respect to the
Teichm\"uller geodesic flow on this locus.
\end{Corollary}

It  remains  to  evaluate  the Lyapunov exponent $\lambda_1^+$ of the
complex  Hodge  bundle  $H^1_{\C}$  over  the  hyperelliptic  locus
$\cQ^{hyp}(1^{2m},-1^{2m})$, which we do in the next section.

% #############################################################
% #############################################################
% #############################################################

\section{Siegel--Veech constants and sum of the Lyapunov exponents
of the Hodge bundle over hyperelliptic loci}
\label{s:Siegel:Veech:constant:for:the:elliptic:locus}

In this section we relate the Siegel--Veech constants of an invariant
hyperelliptic  locus  and  of the underlying stratum of meromorphic
quadratic  differentials with at most simple poles on $\CP$. Applying
the    technique   from~\cite{EKZ}   and   developing   the   results
from~\cite{AEZ}  this  allows  us  to  get  an explicit value for the
desired  Lyapunov  exponent  $\lambda^+_1$  of the hyperelliptic locus and
thus,                             to                            prove
Theorem~\ref{th:lambda1:as:ratio:of:double:factorials}.

%--------------------------------------------------------------------
\subsection{Sum   of   the   Lyapunov   exponents   of   the  complex
Hodge bundle over hyperelliptic loci of quadratic differentials}
\label{ss:sum:of:lambda:plus}

We need the following result

\begin{NNTheorem}[Theorem~2 in~\cite{EKZ}]
\label{theorem:general:quadratic}
Consider  a  stratum $\cQ_1(d_1,\dots,d_\noz)$ in the moduli space of
quadratic   differentials   with   at   most   simple   poles,  where
$d_1+\dots+d_\noz=4g-4$.  Let $\cM_1$ be any regular $\PSL$-invariant
suborbifold of $\cQ_1(d_1,\dots,d_\noz)$.

The  Lyapunov exponents $\lambda_1^+\ge \dots \ge \lambda_g^+$ of
the complex Hodge bundle $H^1_+=H^1_{\C}$ over $\cM_1$
along the Teichm\"uller flow satisfy the following relation:

\begin{equation}
\label{eq:general:sum:of:plus:exponents:for:quadratic}
\lambda^+_1 + \dots + \lambda^+_g
\ = \
\kappa +\frac{\pi^2}{3}\cdot c_{\mathit{area}}(\cM_1)\,,
\end{equation}
where
\begin{equation}
\label{eq:kappa}
\kappa=\cfrac{1}{24}\,\sum_{j=1}^\noz \cfrac{d_j(d_j+4)}{d_j+2}
\end{equation}
and $c_{\textit{area}}(\cM_1)$   is  the  Siegel--Veech  constant
corresponding  to  the  suborbifold $\cM_1$. By convention the sum in
the                 left-hand                 side                 of
equation~\eqref{eq:general:sum:of:plus:exponents:for:quadratic}    is
defined to be equal to zero for $g=0$.
\end{NNTheorem}

In  the  context  of this paper we are particularly interested in the
case  when  $\cM$  is  a  hyperelliptic  locus over some stratum of
meromorphic  quadratic  differentials  with  at  most simple poles in
genus  zero.

Let  $\cM_1$  be  a  closed  $\SL$-  (correspondingly $\PSL$-invariant)
suborbifold  in  a  stratum  of  Abelian  (correspondingly quadratic)
differentials;    let   $\nu$   be   the   associated   $\SL$-ergodic
(correspondingly $\PSL$-ergodic) measure on $\cM_1$.

Consider a locus $\tilde\cM_1$ of all possible double covers of fixed
profile  over  flat surfaces from $\cM_1$. Suppose that it is closed,
connected  and  $\SL$-  (correspondingly  $\PSL$-invariant), and that
$\tilde\nu$  is the ergodic measure on $\tilde\cM$ such that $\nu$ is
the   direct  image  of  $\tilde\nu$  with  respect  to  the  natural
projection $\tilde\cM_1\to\cM_1$.

Let  $c_\cC$  be  an  area  Siegel-Veech  constant  associated to the
counting   of   \textit{multiplicity   one   configuration  $\cC$  of
cylinders}  weighted  by  the  area of the cylinder. (Here the notion
``configuration''  is  understood in the sense of~\cite{Masur:Zorich}
and~\cite{Write:cylinder};  ``multiplicity  one''  means  that  $\cC$
contains  a  single  cylinder). The reader might think of $\cM$ as of
some  stratum of quadratic differentials in genus $0$; then there are
only  two  types  of  configurations (see section~\ref{ss:genus:zero}
below  or  the original paper~\cite{Boissy:configurations}), and both
configurations contain a single cylinder.

We  assume  that  the  profile  of  the  double  cover does not admit
branching  points  inside  the  cylinders of the configuration $\cC$.
Then   the  configuration  $\cC$  induces  a  cylinder  configuration
$\tilde\cC$  on  the  double  cover. Let $\tilde c$ be the associated
area  Siegel--Veech constant. The Lemma below relates $c$ and $\tilde
c$. it is a slight generalization of Lemma~1.1 in~\cite{EKZ}.

\begin{Lemma}
\label{lm:delta:c:area}
If  the  circumference  of  the  cylinder  in  the  configuration has
nontrivial  monodromy  (that  is  if  the lift of the cylinder in the
double  cover  is  a  unique  cylinder twice wider), then $\tilde c =
c/2$.

If  the circumference of the cylinder has trivial monodromy (that is,
if  the  preimage of the cylinder consists in two cylinders isometric
to the one on the base), then $\tilde c = 2c$.
\end{Lemma}

\begin{proof}
The  second  case  corresponds  to Lemma~1.1 in~\cite{EKZ}. The first
case is proved analogously.
\end{proof}

In  the same setting, let $\kappa$ be the expression~\eqref{eq:kappa}
in  degrees  of  singularities  of the flat surfaces in the invariant
manifold  $\cM$ and $\tilde\kappa$ be analogous expression in degrees
of  singularities  of  the  double  covers  of  fixed  profile in the
invariant manifold $\tilde\cM$.

It would be convenient to introduce the following notations
\begin{align*}
&\Dk:=\widetilde{\kappa} - 2 \kappa\\
&\Dc:=c_{\mathit{area}}(\tilde{\cM_1})-2c_{\mathit{area}}(\cM_1)\,.
\end{align*}

\begin{Lemma}
For  any  ramified  double  covering  the  degrees  of  zeroes of the
quadratic  differential  on the underlying surface and of the induced
quadratic  differential  on  the  double  cover satisfy the following
relation:
\begin{equation}
\label{eq:Delta:kappa}
\Dk=\widetilde{\kappa} - 2 \kappa = \frac{1}{4}
\sum_{\substack{\mathit{ramification}\\
\mathit{points}}}
\frac{1}{d_j + 2}\,.
\end{equation}
\end{Lemma}
\begin{proof}
The non ramified zeros cancel out. For the other zeroes (and poles) a
singularity of degree $d_j$ at the ramification point gives rise to a
zero of degree $2d_j + 2$. Hence,
\begin{align*}
\widetilde{\kappa} - 2 \kappa
&=  \frac{1}{24}
\sum_{\substack{\mathit{ramification}\\\mathit{points}}}
\left( \frac{(2d_j+2)(2d_j+6)}{2d_j+4} - 2 \frac{d_j(d_j+4}{d_j+2} \right) \\
&= \frac{1}{12} \sum_{\substack{\mathit{ramification}\\
\mathit{points}}}
\frac{(d_j+1)(d_j+3) - d_j(d_j+4)}{d_j + 2} \\
&= \frac{1}{12}
\sum_{\substack{\mathit{ramification}\\
\mathit{points}}}
\frac{3}{d_j+2}\,.
\end{align*}
\end{proof}

The following notational Lemma would help to simplify certain
bulky computations.

\begin{Lemma}
For  any  locus  $\tilde\cM_1$  of double coverings of fixed profile as
above  over  a  $\PSL$-invariant  orbifold  $\cM_1$  in some stratum of
meromorphic  quadratic  differentials  with  at  most simple poles on
$\CP$ the sum of Lyapunov exponents
$$
\Lambda^+=\lambda^+_1+\dots\lambda^+_g
$$
satisfies the following relation
\begin{equation}
\label{eq:Lambda:plus}
\Lambda^+=\Dk+\frac{\pi^2}{3}\cdot\Dc\,.
\end{equation}
\end{Lemma}
\begin{proof}
For  any  invariant  submanifold  $\cM_1$ in a stratum of meromorphic
quadratic differentials with at most simple poles on $\CP$ the sum of
Lyapunov           exponents           is           null,          so
formula~\eqref{eq:general:sum:of:plus:exponents:for:quadratic} gives
$$
0\ =\ \kappa +\frac{\pi^2}{3}\cdot c_{\mathit{area}}(\cM_1)\,.
$$
By the same
formula~\eqref{eq:general:sum:of:plus:exponents:for:quadratic} gives
we have
$$
\Lambda^+
\ =\ \tilde\kappa +\frac{\pi^2}{3}\cdot c_{\mathit{area}}(\tilde\cM_1)\,.
$$
Extracting  from the latter relation twice the previous one we obtain
the desired relation~\eqref{eq:Lambda:plus}.
\end{proof}

%--------------------------------------------------------------------
\subsection{Configurations   for   the   strata  in  genus  zero  and
corresponding Siegel--Veech constants (after~Boissy
and~Athreya--Eskin--Zorich)}
\label{ss:genus:zero}

In     this     section     we    recall    briefly    the    results
from~\cite{Boissy:configurations}    describing   configurations   of
periodic  geodesics  for flat surfaces in genus zero, and the results
from~\cite{AEZ}   providing   the   values   of   the   corresponding
Siegel--Veech constants. By $Q(d_1,\dots,d_k)$ we denote a stratum of
meromorphic  quadratic differentials with at most simple poles, where
$d_i\in\{-1,1,2,\dots\}$   denote   all   zeroes   and   poles,   and
$\sum_{i=1}^k d_i = -4$.

\textbf{A ``pocket''.}
In  this  configuration  we have a single cylinder filled with closed
regular  geodesics,  such  that  the  cylinder is bounded by a saddle
connection joining a fixed pair of simple poles $P_{j_1}, P_{j_2}$ on
one  side  and  by  a separatrix loop emitted from a fixed zero
$P_i$  of  order  $d_i\ge 1$ on the other side.

By  convention,  the affine holonomy associated to this configuration
corresponds  to  the  closed  geodesic and \textit{not} to the saddle
connection joining the two simple poles. (Such a saddle connection is
twice as short as the closed geodesic.)

\begin{figure}[htb]
%
%  POCKET
   %
\includegraphics{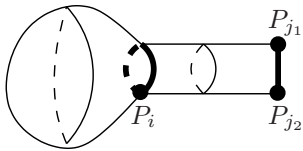}
\begin{picture}(0,0)(-65,-5)
\put(-70,-52){$P_i$}
\put(-18,-16){$P_{j_1}$}
\put(-18,-52){$P_{j_2}$}
\end{picture}
\vspace{60bp}
\caption{
\label{fig:III:pocket}
A  ``pocket''  configuration with a cylinder bounded on one side by a
saddle   connection  joining  two  simple  poles,  and  by  a  saddle
connection joining a zero to itself on the other side.
   }
\end{figure}

By  Theorem~4.5  and  formula (4.28) in~\cite{AEZ}, the Siegel--Veech
constant $c^{pocket}_{j_1,j_2;i}$ corresponding to this configuration
has the form
\begin{equation}
\label{eq:SV:constant:pocket:detailed}
c^{pocket}_{j_1,j_2;i} = \frac{d_i+1}{(k-4)}\cdot\cfrac{1}{2\pi^2}\,.
\end{equation}

One  can consider the union of several configurations as above fixing
the   pair   of  simple  poles  $P_{j_1},  P_{j_2}$  but  considering
\textit{any} zero $P_i$ on the boundary of the cylinder. By Corollary
4.7  and  formula  (4.36)  in~\cite{AEZ}, the resulting Siegel--Veech
constant  $c^{pocket}_{j_1,j_2}$  corresponding to this configuration
has the form
\begin{equation}
\label{eq:SV:constant:pocket}
c^{pocket}_{j_1,j_2} = \cfrac{1}{2\pi^2}\,.
\end{equation}
\smallskip

\textbf{A   ``dumbbell''.}
For  the  second configuration we still have a single cylinder filled
with  closed regular geodesics. But this time the cylinder is bounded
by a separatrix loop on each side. We assume that the separatrix loop
bounding  the cylinder on one side is emitted from a fixed zero $P_i$
of  order  $d_i\ge  1$  and  that  the  separatrix  loop bounding the
cylinder  on  the  other  side  is emitted from a fixed zero $P_j$ of
order $d_j\ge 1$.

Such  a cylinder separates the original surface $S$ in two parts; let
$P_{i_1},  \dots,  P_{i_{k_1}}$  be the list of singularities (zeroes
and  simple  poles)  which get to the first part and $P_{j_1}, \dots,
P_{j_{k_2}}$  be  the list of singularities (zeroes and simple poles)
which  get  to  the  second  part. In particular, we have $i\in\{i_1,
\dots,  i_{k_1}\}$  and $j\in\{j_1, \dots, j_{k_2}\}$. We assume that
$S$  does  not have any marked points. Denoting as usual by $d_k$ the
order  of  the  singularity  $P_k$  we  can  represent  the sets with
multiplicities  $\alpha:=\{d_1,  \dots, d_k\}$ as a disjoint union of
the two subsets
$$
\{d_1, \dots, d_k\}=
\{d_{i_1}, \dots d_{i_{k_1}}\}\sqcup
\{d_{j_1}, \dots, d_{j_{k_2}}\}.
$$
(Recall  that  $\{d_1,  \dots,  d_k\}$ denotes all zeroes and poles.)
This information is considered to be part of the configuration.

\begin{figure}[htb]
\centering
%
% DUMBBELL
   %
\includegraphics{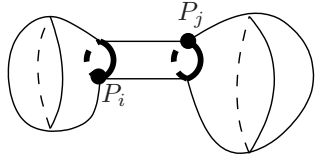}
\begin{picture}(0,0)(45,0)
\put(33,-35){$P_i$}
\put(62,-4){$P_j$}
\end{picture}
\vspace{60bp}
\caption{
\label{fig:IV:dumbbell}
A  ``dumbbell''  composed  of  two flat spheres joined by a cylinder.
Each  boundary  component  of  the  cylinder  is  a saddle connection
joining a zero to itself.
   }
\end{figure}

By  Theorem  4.8  and equation (4.38) in~\cite{AEZ} the corresponding
Siegel--Veech constant $c^{dumbell}$ is expressed as follows:

\begin{equation}
\label{eq:SV:constant:dumbell}
c^{dumbell}_{i,j} =\cfrac{(d_i+1)(d_j+1)}{2}\cdot
\cfrac{(k_1-3)!\,(k_2-3)!}{(k-4)!}\cdot\cfrac{1}{\pi^2}\,.
\end{equation}
According   to~\cite{Boissy:configurations}  and~\cite{Masur:Zorich},
almost any flat surface $S$ in any stratum $\cQ_1(d_1, \dots,d_k)$ of
meromorphic   quadratic  differentials  with  at  most  simple  poles
different  from  the pillowcase stratum $\cQ_1(-1^4)$ does not have a
single  regular  closed  geodesic  not  contained  in  one of the two
families described above.

Finally,  by Corollary 4.10 from~\cite{AEZ} (generalizing the theorem
of   Vorobets   from~\cite{Vorobets})   the   Siegel--Veech  constant
$c_{area}$ is expressed in terms of the above Siegel--Veech constants
as  follows.  For  any  stratum $\cQ_1(d_1,\dots,d_k)$ of meromorphic
quadratic  differentials with simple poles on $\CP$ the Siegel--Veech
constant  $c_{area}$  is  expressed  in  terms  of  the Siegel--Veech
constants of configurations as follows:

\begin{equation}
\label{eq:c:area}
c_{area}=
\cfrac{1}{k-3}\,\cdot\,
\sum_{\substack{\mathit{Configurations}\ \cC\\
\mathit{containing\ a\ cylinder}}}
c_\cC\,.
\end{equation}

%--------------------------------------------------------------------
\subsection{Proof of Theorem~\ref{th:lambda1:as:ratio:of:double:factorials}}
\label{ss:one:complicated:obstacle}

Now      everything     is     ready     for     the     proof     of
Theorem~\ref{th:lambda1:as:ratio:of:double:factorials}.

\begin{proof}
The  connectedness  of  our hyperelliptic locus follows from the fact
that   it   can  be  seen  as  a  $\C^{\ast}$-bundle  over  the  space
$\cC(m,m,4)$  of configurations of points on $\CP$ considered up to a
modular transformation. More precisely, $\cC(m,m,4)$ can be seen as a
set  of  configurations of $2m+4$ distinct points arranged into three
groups  (``colored into three colors'') of cardinalities $\{m,m,4\}$,
where  the  points  inside  each  groups  are  named. The first group
represents  the  simple  zeroes, the second one --- unramified simple
poles,  and the last one --- the ramified simple poles. Clearly, such
space is connected.

By the results of H.~Masur~\cite{Masur} and of W.~Veech~\cite{Veech},
the Teichm\"uller geodesic flow on the underlying stratum is ergodic,
and, moreover, ``sufficiently hyperbolic''. Thus, the induced flow on
any  connected  finite cover is also ergodic. The hyperelliptic locus
$\cQ^{hyp}(1^{2m},-1^{2m})$   over  $\cQ(1^m,-1^{m+4})$  is  a  finite
cover,  and  as  we have just proved it is connected. This proves the
ergodicity of the Teichm\"uller flow on the hyperelliptic locus.

Let     us     compute    $\Delta\tilde    c_{area}^{pocket}$,    see
section~\ref{ss:sum:of:lambda:plus}.   When   both   simple  poles
$P_{j_1},  P_{j_2}$  involved  in  the  ``pocket''  configuration are
unramified points or when they are both ramified, the holonomy of the
hyperelliptic  cover  $\tilde  S\to  S$  along  the  perimeter of the
cylinder   is   trivial,   so   by  Lemma~\ref{lm:delta:c:area}  such
configurations do not contribute to $\Delta\tilde c_{area}^{pocket}$.
By Lemma~\ref{lm:delta:c:area}, a configuration when one of $P_{j_1},
P_{j_2}$ is ramified and the other one is nonramified, contributes to
$\Delta\tilde  c_{area}^{pocket}$ with a weight $-\frac{3}{2}$. Since
we  have  $4$ ramified poles and $m$ nonramified, there are $4m$ such
configurations.      Applying~\eqref{eq:SV:constant:pocket}
and~\eqref{eq:c:area} with $k=m+(m+4)$ we get
\begin{equation}
\label{eq:impact:of:pocket}
\frac{\pi^2}{3}\cdot\Delta\tilde c_{area}^{pocket}=
\frac{\pi^2}{3}\cdot 4m\cdot\left(-\frac{3}{2}\right)\cdot\left(\frac{1}{2m+1}\cdot\frac{1}{2\pi^2}\right)=
-\frac{m}{2m+1}\,.
\end{equation}

Let  us proceed to computation of $\Delta\tilde c_{area}^{dumbbell}$.
When  the  number  of  ramified  simple  poles  on  two  parts of the
``dumbbell''  is  even, the holonomy of the cover along the perimeter
of  the  cylinder  is trivial, so by Lemma~\ref{lm:delta:c:area} such
configurations     do     not     contribute     to     $\Delta\tilde
c_{area}^{dumbbell}$.
By  Lemma~\ref{lm:delta:c:area},  a  configuration when the number of
ramified simple poles on each side of the dumbbell is odd contributes
to  $\Delta\tilde  c_{area}^{dumbbell}$ with a weight $-\frac{3}{2}$.
To  compute  the number of such configurations we remark that we have
to split $m$ named simple zeroes into two groups of $m_1$ and $m-m_1$
ones;  we  also  have to split $4$ ramified simple poles into $1$ and
$3$;  finally  we  have  to  split  $m$  unramified simple poles into
$m_1+1$  and  $m-m_1-1$ ones to have in total $m_1+2$ simple poles on
one  side  of  the ``dumbbell'' and $m-m_1+2$ on the other side. Note
that the fact that there is a single ramified simple pole on one part
and  $3$  ramified  simple  poles  on  the  other  makes our count of
configurations  asymmetric. Finally note that we have to chose one of
$m_1$  zeroes  to  be  located at the boundary of the cylinder on one
side  and  one of $m-m_1$ zeroes to be located at the boundary of the
cylinder on the other side.

For any given $m_1$, where $1\le m_1 \le m-1$ our count gives
$$
\binom{m}{m_1}\cdot\binom{4}{1}\cdot\binom{m}{m_1-1}\cdot \left(m_1\cdot(m-m_1)\right)
$$

Applying                          the                         general
formulae~\eqref{eq:SV:constant:dumbell}
and~\eqref{eq:c:area}  to  $c^{dumbell}_{i,j;\,  area}$;  taking  into
consideration  that  in  our  particular  case  we  have $d_i=d_j=1$;
$k_1=m_1+(m_1+2)$;   $k_2=(m-m_1)+(m-m_1+2)$;  and  $k=m+(m+4)$;  and
applying~\eqref{eq:c:area} we get

\begin{multline}
\label{eq:SV:const:dumbell:swiss:cross}
\frac{\pi^2}{3}\Delta\tilde c^{dumbell}_{area}
=\\=
\frac{\pi^2}{3}\sum_{m_1=1}^{m-1}
\Bigg(\binom{m}{m_1}\cdot\binom{4}{1}\cdot\binom{m}{m_1-1}\cdot \left(m_1\cdot(m-m_1)\right)\Bigg)
\cdot\left(-\frac{3}{2}\right)
\cdot\\ \cdot
\frac{1}{(2m+4)-3}\cdot
\cfrac{(1+1)(1+1)}{2}\cdot
\cfrac{\left((2m_1+2)-3\right)!\cdot\left((2m-2m_1+2)-3\right)!}{(2m+4)-4)!}\cdot\cfrac{1}{\pi^2}
=\\=
-\frac{1}{2m+1}\cdot
\sum_{m_1=1}^{m-1}
\binom{m}{m_1}\binom{m}{m_1-1}\cdot
\cfrac{(2m_1)!\cdot(2m-2m_1)!}{(2m)!}
=\\=
-\frac{1}{2m+1}\cdot
\sum_{m_1=1}^{m-1}
\cfrac{\binom{m}{m_1}\binom{m}{m_1-1}}{\binom{2m}{2m_1}}
\end{multline}

Summing                                up~\eqref{eq:impact:of:pocket}
and~\eqref{eq:SV:const:dumbell:swiss:cross} and applying the standard
convention
$$
\binom{m}{m+1}:=0\qquad\text{and}\qquad\binom{m}{0}:=1
$$
we obtain
\begin{equation}
\label{eq:delta:tilde:c:area}
\frac{\pi^2}{3}\Delta\tilde c_{area}=
-\frac{1}{2m+1}\sum_{m_1=0}^{m}
\frac{\binom{m}{m_1}\binom{m}{m_1+1}}
{\binom{2m}{2m_1}}
=-1+\frac{(2m)!!}{(2m+1)!!}
\,.
\end{equation}
where     the     second     equality     ils    the    combinatorial
identity~\eqref{eq:00:01:00}                 proved                in
Proposition~\ref{prop:main:identity:1}.

By formula~\eqref{eq:Delta:kappa} we have
\begin{equation}
\label{eq:Delta:tilde:kappa:1}
\Delta\tilde\kappa=\frac{1}{4}\cdot\sum_{i=1}^4\frac{1}{(-1+2)}=1
\end{equation}

Plugging           the          results~\eqref{eq:delta:tilde:c:area}
and~\eqref{eq:Delta:tilde:kappa:1}  of our calculation in the general
formula~\eqref{eq:Lambda:plus} for the sum $\Lambda^+$ we get
$$
\Lambda^+=1+\left(-1+\frac{(2m)!!}{(2m+1)!!}\right)\,.
$$
It remains to note that the stratum $\cQ(1^{2m},-1^{2m})$ corresponds
to  genus  one,  so  the  spectrum  of  Lyapunov exponents of $H^1_+$
contains  a  single  entry  and $\Lambda^+=\lambda^+_1$, which proves
Theorem~\ref{th:lambda1:as:ratio:of:double:factorials}.
\end{proof}

\begin{Corollary}
\label{cor:lambda1:tends:to:zero}
The  Lyapunov exponent $\lambda_1^+(m)$ tends to zero as $m$ tends to
infinity. More precisely,
\begin{equation}
\label{eq:asymptotics:for:lambda1}
\delta(m)=
\frac{\sqrt{\pi}}{2\sqrt{m}}
\left(1+O\left(\frac{1}{m}\right)\right)\,.
\end{equation}
\end{Corollary}
\begin{proof}
Rewriting  the  double  factorials  in  terms  of usual factorials as
$$
\frac{(2m)!!}{(2m+1)!!}=
\frac{\left((2m)!!\right)^2}{(2m+1)!}=
\frac{1}{2m+1}\cdot 2^{2m}\frac{(m!)^2}{(2m)!}\,,
$$
and applying the Stirling's formula
$$
n!=\sqrt{2\pi n}\left(\frac{n}{e}\right)^n
\left(1+O\left(\frac{1}{n}\right)\right)\,,
$$
to  both  factorials  and simplifying the resulting expression we get
$$
\frac{(2m)!!}{(2m+1)!!}=
\frac{\sqrt{\pi m}}{2m+1}
\left(1+O\left(\frac{1}{m}\right)\right)\,.
$$
Clearly,  the  latter  expression  tends  to  zero  as  $m$  tends to
infinity.
\end{proof}

% #############################################################
% #############################################################
% #############################################################

\section{Combinatorial identities}
\label{s:combinatorial:identities}

\begin{Proposition}
\label{prop:main:identity:1}
For any $m\in\N$ the following identities hold
\begin{align}
\label{eq:00:00:00}
&\sum_{k=0}^{m}
\cfrac{\binom{m}{k}\binom{m}{k}}{\binom{2m}{2k}}
=
\cfrac{(2m)!!}{(2m-1)!!} = 4^m \cfrac{(m!)^2}{(2m)!}\\
\notag\\
\label{eq:00:01:00} % \label{eq:main:identity:1}
&\sum_{k=0}^{m}
\frac{\binom{m}{k}\binom{m}{k+1}}
{\binom{2m}{2k}}
=
2m+1-\frac{(2m)!!}{(2m-1)!!}\\
\notag\\
\label{eq:00:11:00}
&\sum_{k=0}^{m}
\cfrac{\binom{m}{k}\binom{m+1}{k+1}}{\binom{2m}{2k}}
=2m+1\,.
\end{align}
\end{Proposition}

\begin{proof}[Proof of Proposition~\ref{prop:main:identity:1}]
First note that
$$
\binom{m+1}{k+1}=\binom{m+1}{k}+\binom{m}{k}\,.
$$
Thus,  the expression in the left-hand side of~\eqref{eq:00:11:00} is
the  sum  of  the  corresponding  expressions  in the left-hand sides
of~\eqref{eq:00:00:00} and~\eqref{eq:00:01:00}. Hence, any two out of
three  identities~\eqref{eq:00:00:00}--\eqref{eq:00:11:00}  imply the
remaining one.
\smallskip

\noindent  \textbf{Proof of identity~(\ref{eq:00:11:00}).}
Developing       binomial      coefficients      into      factorials
in~\eqref{eq:00:11:00}  and  moving  the  common  factorials  to  the
right-hand side, we can rewrite this identity as
\begin{equation}
\label{eq:s3:definition}
\sum_{k=0}^{m}
\frac{(2k)!}{k!\,(k+1)!}\cdot
\frac{(2m-2k)!}{(m-k)!\,(m-k)!}\stackrel{?}{=}
\binom{2m+1}{m}\,.
\end{equation}
Denote the sum in the left-hand side of~\eqref{eq:s3:definition}
by $s_3(m)$.
We  prove  the latter identity  by induction. For $m=0$ it clearly
holds.  Assuming  that~\eqref{eq:s3:definition} holds for some integer
$m$ we are going to prove that
\begin{equation}
\label{eq:relation:for:s3}
(m+2)\cdot s_3(m+1)-2(2m+3)\cdot s_3(m)=0\,.
\end{equation}
Since  the  right-hand side of~\eqref{eq:s3:definition} satisfies the
same  relation  (which  is  an  exercise), this completes the step of
induction. It remains to prove~\eqref{eq:relation:for:s3} under
assumption~\eqref{eq:s3:definition}.
\begin{multline*}
(m+2)\cdot s_3(m+1)-2(2m+3)\cdot s_3(m)
=\\=
(m+2)\sum_{k=0}^{m+1}
\frac{(2k)!}{k!\,(k+1)!}\cdot
\frac{(2m+2-2k)!}{(m+1-k)!\,(m+1-k)!}
-\\-
2(2m+3)\sum_{k=0}^{m}
\frac{(2k)!}{k!\,(k+1)!}\cdot
\frac{(2m-2k)!}{(m-k)!\,(m-k)!}
=\\=
(m+2)\frac{(2m+2)!}{(m+1)!(m+2)!}+
\sum_{k=0}^{m}
\frac{(2k)!}{k!\,(k+1)!}\cdot
\frac{(2m-2k)!}{(m-k)!\,(m-k)!}
\cdot\\ \cdot\left(
(m+2)\frac{(2m+2-2k)(2m+1-2k)}{(m+1-k)(m+1-k)}-2(2m+3)
\right)
=\\=
2\frac{(2m+1)!}{(m+1)!\,m!}+
\sum_{k=0}^{m}
\frac{(2k)!}{k!\,(k+1)!}\cdot
\frac{(2m-2k)!}{(m-k)!\,(m-k)!}
\cdot\left(-2\frac{k+1}{m+1-k}\right)
=\\=
2\binom{2m+1}{m}-2\sum_{k=0}^{m}
\frac{(2k)!}{k!\,k!}\cdot
\frac{(2m-2k)!}{(m-k)!\,(m-k+1)!}
=\\=
2\binom{2m+1}{m}-2\sum_{j=0}^{m}
\frac{(2(m-j))!}{(m-j)!\,(m-j)!}\cdot
\frac{(2j)!}{j!\,(j+1)!}
=\\=
2\binom{2m+1}{m}-2s_3(m)=0\,.
\end{multline*}
where    the    last    equality   is   the   induction   assumption.
Identity~\eqref{eq:s3:definition}              and              hence
identity~\eqref{eq:00:11:00} is proved.
\smallskip

\noindent
\textbf{Proof of identity~(\ref{eq:00:00:00}).}
Developing       binomial      coefficients      into      factorials
in~\eqref{eq:00:00:00} and simplifying common factorials in the right
and  left  hand  side  of~\eqref{eq:00:00:00}, we  can  rewrite  this
identity as
$$
\sum_{j=0}^m \binom{2j}{j}\binom{2m-2j}{m-j}=4^m\,,
$$
which is identity (3.90) in~\cite{Gould}.

Proposition~\ref{prop:main:identity:1} is proved.
\end{proof}

\appendix

\section{Removing some squares in the wind-tree model}
\label{a:removing:squares:in:the:wind:tree}

Let us consider the periodic wind-tree models with square obstacles.
If we remove periodically one obstacle out of four,
we can still perform the same construction as before
and end up with a
surface in $\cQ^{hyp}(1^6,-1^6)$ over $\cQ(1^3, -1^7)$.

\begin{figure}[htb]
%
%  Wind-tree
   %
\includegraphics{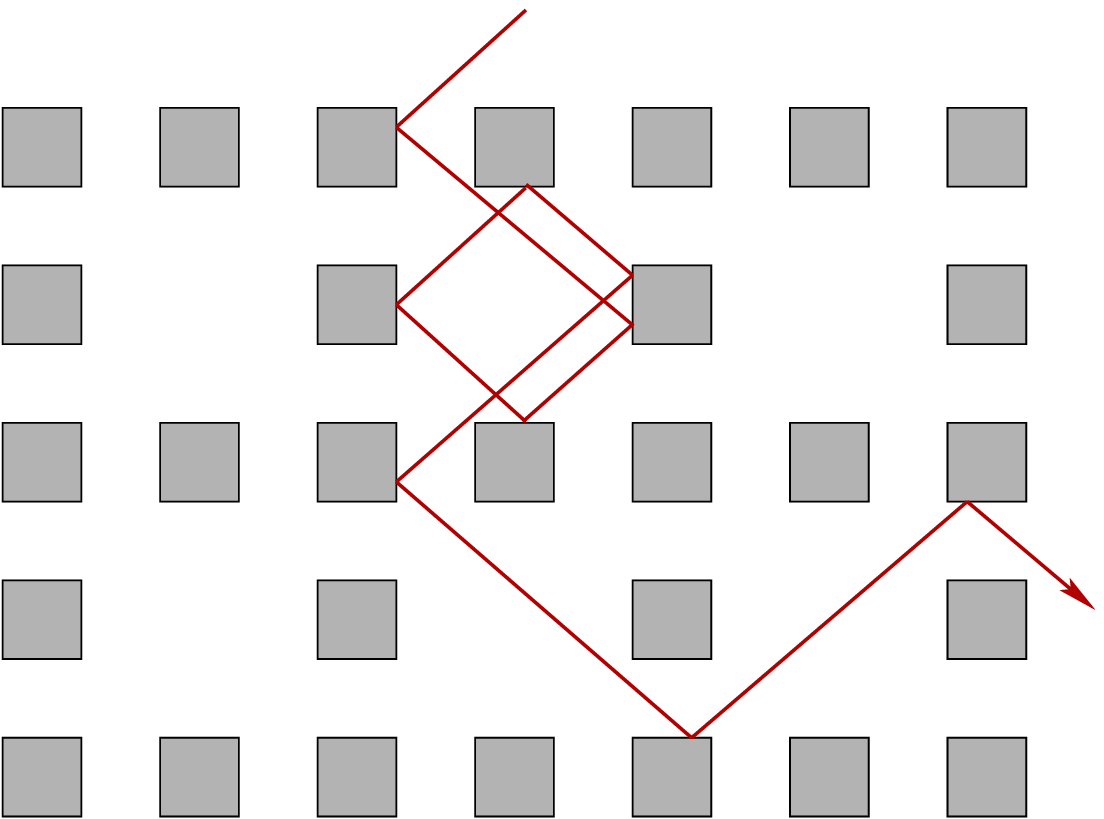}
\vspace{120bp}
\caption{
\label{fig:chessboard}
A square-tiled wind-tree where we regularly remove
one obstacle in every repetitive pattern of four.
   }
\end{figure}

Namely, unfolding the wind-tree in Figure~\ref{fig:chessboard}
and taking the quotient over $\Z\oplus\Z$ we get a compact translation
surface represented in Figure~\ref{fig:chessboard:flat:surface}. This
translation surface corresponds to the
wind-tree in Figure~\ref{fig:chessboard}
exactly in the same way as the compact flat surface
in Figure~\ref{fig:windtree:surface:genus:5} corresponds to the
original wind-tree in Figure~\ref{fig:windtree}.

\begin{figure}[htb]
%\centering
%
%  Four chessboards
   %
\includegraphics{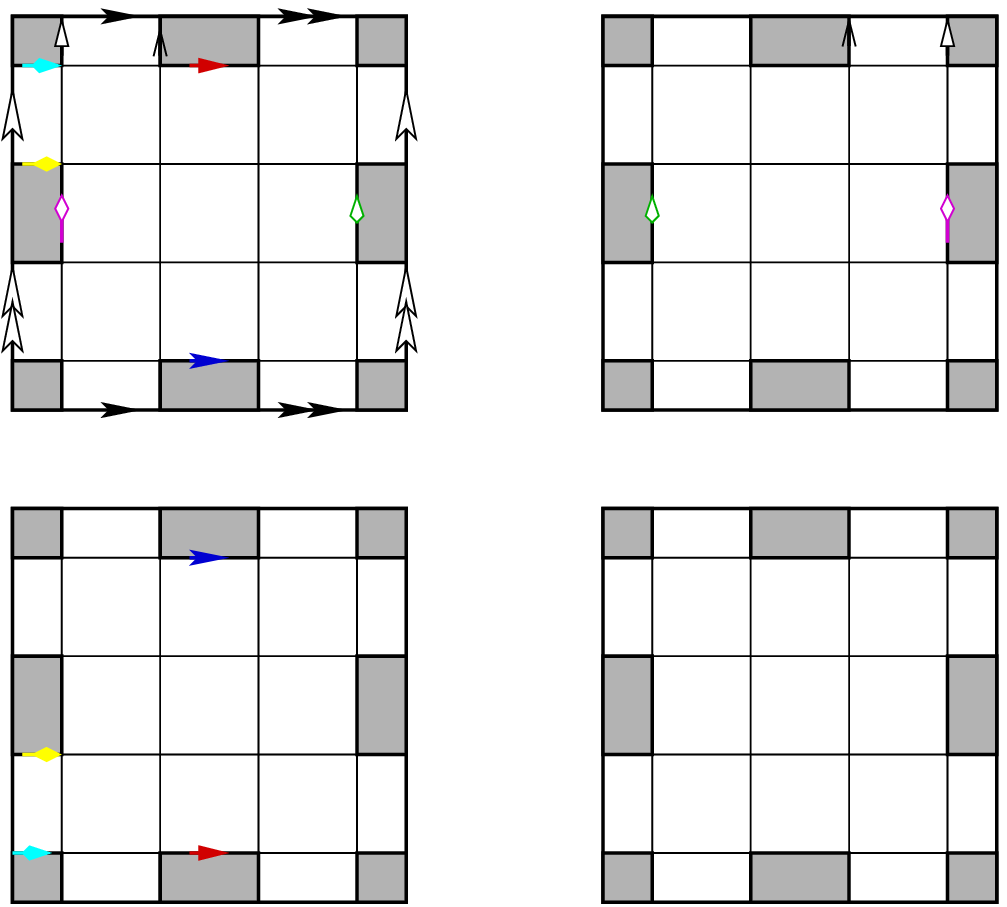}
   %
% \begin{picture}(0,0)(28,11)
% \put(-32,-110){$h_{00}$}
% \put(-32,4){$h_{01}$}
% \put(55,-110){$h_{10}$}
% \put(55,4){$h_{11}$}
% %
% \put(-76,-78){$v_{00}$}
% \put(98,-78){$v_{10}$}
% \put(-76,-25){$v_{01}$}
% \put(98,-25){$v_{11}$}
% \end{picture}
   %
\vspace{180bp}
\caption{
\label{fig:chessboard:flat:surface}
The translation surface $X\in\cH(2^{12})$
obtained as the quotient over $\Z\oplus\Z$ of an
unfolded wind-tree billiard table in the left
picture in Figure~\ref{fig:chessboard}.
   }
\end{figure}

As in the previous examples,
the  flat    surface    $X$ in
Figure~\ref{fig:chessboard:flat:surface}  has the group $(\Z/2\Z)^3$
as  a  group  of isometries (we have already seen exactly the same
group of symmetries for the surface in
Figure~\ref{fig:windtree:surface:genus:5}). As before, we can choose
as generators the  isometries $\tau_h$ and $\tau_v$ interchanging the
pairs of flat tori with holes in the same rows (correspondingly
columns) by parallel  translations and the isometry $\iota$ acting on
each of the four  tori  with holes as the central symmetry with the
center in the center of the square. Passing to the quotient over
$(\Z/2\Z)$ spanned by $\tau_v$ we get the square-tiled surface
$\hat S\in\cH(2^6)$ as
in Figure~\ref{fig:chessboard:two:tori}.

It is encoded by the following two permutations
\[\scriptsize
\begin{array}{l}
r = (1,2,16,14,15,3)(4,5,6,7)(8,22)(9,21)(10,11,12,13)(14,15,16)(17,18,19,20)(23,24,25,26) \\
u = (1,4,10)(2,5,9,11)(3,7,8,13)(6,12)(14,17,23)(15,18,22,24)(16,20,21,26)(19,25).
\end{array}
\]

We construct its $\operatorname{SL}(2,{\mathbb R})$-orbit;
analyze its elements, and apply
formula~(2.12) from~\cite{EKZ} to find
\begin{equation}
\label{eq:seven:exponents}
\lambda_1(\hat S)+\dots+\lambda_7(\hat S)=\frac{3088}{1053}
%\frac{491}{1053}\simeq
\end{equation}
as the sum of the Lyapunov exponents of the resulting
arithmetic Teichm\"uller curve $\cL(\hat S)$.

\begin{figure}[htb]
%\centering
%
%  Two chessboards
   %
\includegraphics{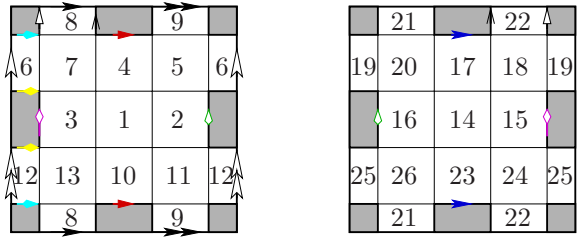}
\begin{picture}(0,0)(0,0)
\begin{picture}(0,0)(0,0)
\put(-85,-12){$8$}
\put(-45,-12){$9$}
\put(-102,-29){$6$}
\put(-85,-29){$7$}
\put(-65,-29){$4$}
\put(-45,-29){$5$}
\put(-28,-29){$6$}
\put(-85,-49){$3$}
\put(-65,-49){$1$}
\put(-45,-49){$2$}
\put(-105,-70){$12$}
\put(-89,-70){$13$}
\put(-68,-70){$10$}
\put(-47,-70){$11$}
\put(-31,-70){$12$}
\put(-85,-86){$8$}
\put(-45,-86){$9$}
\end{picture}
 %==========================
\begin{picture}(0,0)(-120,0)
\put(-85,-12){$21$}
\put(-42,-12){$22$}
\put(-100.5,-29){$19$}
\put(-85,-29){$20$}
\put(-63,-29){$17$}
\put(-43,-29){$18$}
\put(-26,-29){$19$}
\put(-85,-49){$16$}
\put(-63,-49){$14$}
\put(-43,-49){$15$}
\put(-100.5,-70){$25$}
\put(-85,-70){$26$}
\put(-63,-70){$23$}
\put(-43,-70){$24$}
\put(-26,-70){$25$}
\put(-85,-86){$21$}
\put(-42,-86){$22$}
\end{picture}
\end{picture}
\vspace{90bp}
\caption{
\label{fig:chessboard:two:tori}
The  flat  surface $\hat S=X/\tau_v$.
   }
\end{figure}

Now, following the strategy of section~\ref{ss:Original:wind:tree:revisited}
we pass to the second quotient
$\tilde S=X/\langle\tau_h,\tau_v\rangle\in\cQ^{hyp}(1^6,-1^6)$
(on the left of Figure~\ref{fig:chessboard:one:torus:and:a:half})
followed by the quotient
$S=X/\langle\tau_h,\tau_v,\iota\rangle\in\cQ(1^3,-1^{3+4})$
(on the right of Figure~\ref{fig:chessboard:one:torus:and:a:half}).

\begin{figure}[htb]
%\centering
%
%  One chessboards and its half.
   %
\includegraphics{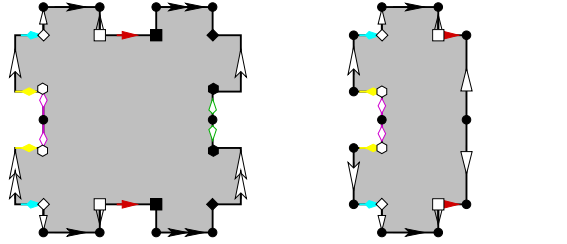}
\vspace{90bp}
\caption{
\label{fig:chessboard:one:torus:and:a:half}
The  flat  surfaces
$S=X/\langle\tau_h,\tau_v\rangle\in\cQ^{hyp}(1^6,-1^6)$
(on the left) and
$S=X/\langle\tau_h,\tau_v,\iota\rangle\in\cQ(1^3,-1^{3+4})$
(on the right). We inverse shadowing at this picture
with respect to Figures~\ref{fig:chessboard:two:tori}
and~\ref{fig:chessboard:flat:surface}: now we shadow the surface,
and not the obstacles. Tiny black discs
at the vertices represent the simple poles; the other vertices
correspond to simple zeroes.
   }
\end{figure}

The surface $\hat S\in\cH(2^6)$ is the orienting double cover
of the surface $\tilde S$. Hence, the Lyapunov spectrum~\eqref{eq:seven:exponents}
of the arithmetic Teichm\"uller curve $\cL(\hat S)$ is the union
of the spectra of Lyapunov exponents
of the arithmetic Teichm\"uller curve $\cL(\tilde S)$:
$$
\{\lambda_1(\hat S),\dots,\lambda_7(\hat S)\}=
\{\lambda_1^-(\tilde S),\dots,\lambda_6^-(\tilde S)\}\cup\{\lambda_1^+(\tilde S)\}\,.
$$
By formula~(2.4) from~\cite{EKZ} we have
$$
\big(\lambda_1^-(\tilde S)+\dots+\lambda_6^-(\tilde S)\big)-\lambda_1^+=
\cfrac{1}{4}\,\cdot\,\sum_{\substack{j \text{ such that}\\
d_j \text{ is odd}}}
\cfrac{1}{d_j+2}=
\frac{1}{4}\cdot 6\cdot\left(\frac{1}{3}+1\right)=2\,.
$$
Together with~\eqref{eq:seven:exponents} this implies that
$$
\lambda_1^+(\tilde S)=\frac{1}{2}\left(\frac{3088}{1053}-2\right)=\frac{491}{1053}.
$$
and, hence, that the diffusion rate for the square-tiled wind-tree as in
in Figure~\ref{fig:chessboard} is given by
$$
\delta=\lambda_1^+(\tilde S)=\frac{491}{1053}.
$$

\medskip

We shall see that in all the other cases, removing some of the obstacles out of
every repetitive block of $2\times 2$ obstacles
in the wind-tree model with any parameters $a$ and $b$ we keep the
diffusion rate $\delta=\lambda_1^+ = 2/3$.
There are three cases to consider:
\begin{enumerate}
\item removing two obstacles which are in the same row or in the same column;
\item removing two obstacles which are on the same diagonal;
\item removing three obstacles.
\end{enumerate}
In the first and in the last case we can just choose a new
fundamental domain of the rectangular lattice (duplicating it
in the first case and choosing it $2\times 2$ bigger in the third case)
to reduce the situation to the original wind-tree with different
parameters. We have seen that the diffusion rate in the original wind-tree
as in Figure~\ref{fig:windtree} does not depend neither on the parameters
of the lattice, nor on the parameters of the obstacle.
Hence in both cases $\delta=\lambda_1^+ =2/3$.

Let us consider the case where we remove two obstacles on the same diagonal
as in Figure~\ref{fig:chessboard:2}.
In this case, we can modify the construction a little bit. We unfold the billiard
as before, but then we quotient the resulting periodic surface by the integer sublattice
spanned by the integer vectors $(1,1)$ and $(1,-1)$ in $\ZZ\oplus\ZZ$
and not buy the entier lattice $\ZZ\oplus\ZZ$ as before.
The quotient belongs to the same locus as
the original wind-tree. We deduce that we again have $\delta=\lambda_1^+ = 2/3$.

Let us emphasize that the above construction is very specific to
the sublattice $\ZZ (1,1) \oplus \ZZ (1,-1)$ which is invariant under reflexions
by the horizontal and vertical axes. In such situation the quotient keeps a $(\ZZ/2\ZZ)^3$ group
of symmetry.

\begin{figure}[htb]
%
%  Wind-tree: 2 out of 4
   %
\includegraphics{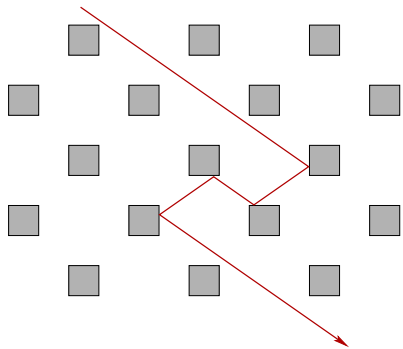}
\vspace{115bp}
\caption{
\label{fig:chessboard:2}
A square-tiled wind-tree where we regularly remove
every second square obstacle has the original diffusion
rate $\delta=2/3$.
   }
\end{figure}

%====================================================================
\section{What is next?}
\label{a:What:is:next}

The  name  of the conference and personality of Bill~Thurston suggest
to  discuss  what are the current directions of research in the area.
We  focus  on  those aspects which are somehow related to the current
paper.  The  selection  represents  our particular taste and might be
subjective.  These  problems  are, certainly, known in our community;
some  of  them are already subjects of intensive investigation, so we
do not claim novelty in this discussion.

The  Magic  Wand  Theorem  of Eskin--Mirzakhani--Mohammadi enormously
amplified the importance of the classification of the orbit closures.
Ideally, any problem about an individual flat surface (say, diffusion
rate  for a wind-tree model with obstacles of some specific ``rational
shape'') should be solved as follows: touch the corresponding surface
with  the  Magic Wand and find the corresponding orbit closure $\cL$.
Now touch the complex Hodge bundle over $\cL$ with the Magic Wand and
find  its  irreducible  component containing the two integer cocycles
responsible  for  the  diffusion.  Find or estimate the corresponding
Lyapunov  exponents  and  the problem is solved. This strategy evokes
three problems: how to find an orbit closure? How to find irreducible
components  of  the  Hodge  bundle?  How to estimate the top Lyapunov
exponent of a given invariant flat subbundle?

%--------------------------------------------------------------------
\subsection{Classification of invariant suborbifolds.}

\begin{Problem}
Classify  all  $\GLR$-invariant suborbifolds in $\cH(d_1,\dots,d_n)$.
\end{Problem}

The  invariant  suborbifolds in the strata of quadratic differentials
are not mentioned since replacing any flat surface corresponding to a
quadratic  differential  by  its  canonical  ramified double cover on
which the induced quadratic differential becomes a global square of a
holomorphic  1-form we get an associated invariant suborbifold in the
corresponding stratum of Abelian differentials.

In genus $2$ the problem is solved by C.~McMullen~\cite{McMullen}.
There are very serious advances in this problem for
the stratum $\cH(4)$, due to
Aulicino, Nguyen and Wright~\cite{Aulicino:Nguyen:Wright} for $\cH^{odd}(4)$
and Nguyen and Wright~\cite{Aulicino:Nguyen:Wright} for $\cH^{hyp}(4)$.
Certain finiteness results for the number of
primitive Teichm\"uller curves are obtained by
Bainbirdge and M\"oller~\cite{Bainbirdge:Moeller}
and by Matheus and Wright~\cite{Matheus:Wright}.
Situation in the Prym loci
is very well understood due to the work of Lanneau and Nguyen~\cite{Lanneau:Nguyen:1},
\cite{Lanneau:Nguyen:2} and~\cite{Lanneau:Nguyen:3}.

The   Theorem   of   Eskin--Mirzakhani--Mohammadi  states  that  such
suborbifolds  correspond to complex affine subspaces in cohomological
coordiantes.           Due           to           result           of
Avila--Eskin--M\"oller~\cite{Avila:Eskin:Moller}  the  projection  of
such  affine  subspace  to absolute cohomology should be a symplectic
subspace:  the  restriction of the intersection form to this subspace
is   nondegenerate.  Results  of  A.~Wright~\cite{Write:field:of:def}
impose conditions on the field of definition. At some point all these
constraints  started  to  seem  so restrictive that there was a doubt
whether any really new invariant suborbifolds are left? Namely, there
is  a  separate  question  of  classification of Teichm\"uller curves
(both  arithmetic and non arithmetic), and there are plenty invariant
suborbifolds  which can be obtained from Teichm\"uller curves or from
connected  components  of  the  strata  by  various ramified covering
constructions.  The question was whether there is anything else (with
exception for some sporadic series of invariant suborbifolds specific
for very small genera).

A  recent  example of an invariant submanifold of enigmatic origin in
$\cH(6)$  (which  does  not  fit  any  of the known schemes mentioned
above) was recently found by M.~Mirzakhani and A.~Wright. For a brief
overview    of   the   state   of   the   art   in   this   direction
see~\cite{Write:Oberwolfach}.

%--------------------------------------------------------------------
\subsection{Classification    of    invariant   subbundles   of   the
complex Hodge bundle}

The  complex  Hodge bundle $H^1_{\C}$ over the moduli space $\cM_g$
of  curves  has  the homology space $H^1(C;\C)$ as a fiber over the
point represented by the curve $C$. This bundle can be pulled back to
any stratum $\cH(m_1,\dots,m_n)$.

\begin{Problem}
Find  the decomposition of the complex Hodge bundle $H^1_\C$ over
any    given    $\GLR$-invariant    suborbifold    into   irreducible
$\GLR$-equivariant subbundles.
\end{Problem}

The    fact   that   such   decomposition   exists   is   proved   by
S.~Filip~\cite{Filip:semisimplicity}.  Here ``irreducible'' should be
understood  in the sense that there is no further splitting even when
we  pass  to  any  finite  ramified  cover  of  the  $\GLR$-invariant
suborbifold.

The  problem  is  meaningful already for the strata! It is absolutely
frustrating,  but we do not have a proof that the only equivariant
subbundles  of $H^1_{\C}$ over a connected component of any stratum
is  the  tautological  subbundle and its symplectic complement. We do
not  know  either whether $H^1_+$ over any connected component of any
stratum of quadratic differentials is irreducible in this sense.

The  current tools allow, in principal, to prove the latter two facts
more  or  less  by hands for some low-dimensional strata. Namely, one
can  start  with  an  arithmetic Teichm\"uller curve, compute certain
number    of    monodromy    matrices    and   then   use   technique
of~\cite{Matheus:Moeller:Yoccoz}   or   the   method   of   Eskin  as
in~\cite{Forni:Matheus:Zorich} to prove irreducibility of the complex
Hodge  bundle  over  the  Teichm\"uller  curve. As a consequence
of~\cite{Filip:semisimplicity} one gets the irreducibility over
over  the  ambient  stratum.  However,  what is really needed is some
general proof for all strata at once.

A related question is

\begin{Problem}
What  groups are realizable as Zariski closures of leafwise monodromy
groups  of equivariant irreducible blocks of the complex Hodge bundle
$H^1_{\C}$ over $\GLR$-invariant suborbifolds?
\end{Problem}

The                 original                 guess                 of
Forni--Matheus--Zorich~\cite{Forni:Matheus:Zorich}  states  that this
group is always $\operatorname{SU}(p,q)$ for appropriate $p$ and $q$.
The   paper   of  S.~Filip~\cite{Filip:Zariski}  shows  that  general
Hodge-theoretical  arguments  admit  a  priori larger list (including
some more sophisticated representations of $\operatorname{SU}(p,q)$).
However,  it is not clear which of these groups (representations) are
realizable  as  Zariski  closures  of  leafwise  monodromy  groups of
equivariant   subbundles   of   the   complex   Hodge   bundle   over
\textit{$\GLR$-invariant  suborbifolds}  (and not just over some flat
subbundles  of the complex Hodge bundle over abstract submanifolds of
the moduli space).

%--------------------------------------------------------------------
\subsection{Estimates for individual Lyapunov exponents}

Paper~\cite{EKZ}  provides  a  formula  for  the  sum of the positive
Lyapunov  exponents  of  the  complex  Hodge  bundle  over  along the
Teichm\"uller  geodesic  flow.  Though there is no reason to hope for
exact  values  of  individual Lyapunov exponents, paper~\cite{Fei} of
Fei~Yu  conjectures  that partial sums of Lyapunov exponents might be
estimated  through  Chern  classes of holomorphic vector bundles over
Teichm\"uller  curves  normalized  by  the  Euler  characteristics of
Teichm\"uller curves. More generally (though less precisely):

\begin{Problem}
Study  extremal  properties  of  the  ``curvature''  of  the Lyapunov
subbundles  compared  to  holomorphic subbundles of the Hodge bundle.
Estimate the individual Lyapunov exponents.
\end{Problem}

For  example, estimates for $\lambda^+_1$ over hyperelliptic locus in
the principal stratum and estimates for $\lambda^+_1$ over the entire
principal    stratum    $Q(1,\dots,1)$   of   holomorphic   quadratic
differentials  would  provide  estimates  for  the  diffusion rate in
certain families of wind-tree billiards.

%--------------------------------------------------------------------
\subsection{Siegel--Veech constants in terms of an adequate intersection theory}

The  sum  of  positive Lyapunov exponents of the complex Hodge bundle
over   an   $\SLR$-invariant   suborbifold   $\cL$   in   a   stratum
$\cH_1(m_1,\dots,m_n)$ is expressed in~\cite{EKZ} as
$$
\lambda_1+\cdots+\lambda_g=
\frac{1}{12}\sum_{i=1}^n \frac{m_i(m_i+2)}{m_i+1}+
\frac{\pi^2}{3}\cdot c_{\mathit{area}}(\cL)\,,
$$
where  $c_{\mathit{area}}(\cL)$  is  the  Siegel--Veech  constant  of
$\cL$.  Currently there are two formulae for $c_{\mathit{area}}(\cL)$
for  two extremal cases of $\cL$. When $\cL$ is a connected component
of  a stratum, the Siegel--Veech constant $c_{\mathit{area}}(\cL)$ is
expressed  as  a  polynomial  in volumes of simpler \textit{principal
boundary  strata}  (normalized by the volume of the initial stratum);
see~\cite{Eskin:Masur:Zorich} for the strata of Abelian differentials
and~\cite{Goujard}  for  the  strata of quadratic differentials. When
$\cL$ is a Teichm\"uller curve, $c_{\mathit{area}}(\cL)$ is expressed
as  the  integral  of  the Chern class of the determinant bundle over
$\cL$    normalized   by   the   Euler   characteristic   of   $\cL$,
see~\cite{Kontsevich},   \cite{Bouw:Moeller},   \cite{EKZ:JMD}.   The
challenge is to construct a bridge between these two cases:

\begin{Problem}
Express   $c_{\mathit{area}}(\cL)$   in   terms   of  an  appropriate
intersection theory.
\end{Problem}

Here  we  do not mean some kind of asymptotic limit formulae, but
something    in    the    spirit    of   ELSV-formula   for   Hurwitz
numbers, see~\cite{ELSV}.

There  is  certain  resemblance  between  the  ``hyperbolic  regime''
studied
in~\cite{Mirzakhani:volumes:and:intersection:theory}--\cite{Mirzakhani:grouth:of:simple:geodesics}
by M.~Mirzakhani
and   the   ``flat   regime''  studied  in~\cite{Eskin:Masur:Zorich}.
M.~Mirzakhani  used  dynamics  on  moduli  space to relate the length
functions   of   simple  geodesics  on  hyperbolic  surfaces  to  the
Weil--Peterson  volumes of the moduli spaces $\cM_{g,n}$ of punctured
Riemann  surfaces,  and  also to relate the Weil--Peterson volumes to
the   intersection   numbers   of   tautological  line  bundles  over
$\cM_{g,n}$.

Morally,  we have somehow similar situation in a parallel flat world.
The  step of relating the counting functions for simple \textit{flat}
geodesics  (for  the  flat metrics in the same conformal class as the
original  hyperbolic  metric)  to polynomial in volumes of the strata
with  respect  to  Masur--Veech  volume  form  is  already  performed
in~\cite{Eskin:Masur:Zorich}  and~\cite{Goujard}. The challenge is to
accomplish the second step and to relate these volumes to an adequate
intersection theory.

Actually, certain parallel between hyperbolic and flat word manifests
in  further  aspects.  For  example,  there  are  conjectural  simple
asymptotic   formulae   for   large  genera  for  the  Weil--Peterson
volumes~\cite{Mirzakhani:Zograf}  and for Masur--Veech volumes. About
ten  years  ago  A.~Eskin  and  one of the authors conjectured a very
simple  and  explicit asymptotic formula for the Masur--Veech volume.
For   the  principal  stratum  $\cH(1^{2g-2})$  this  conjecture  was
recently    proved    in~\cite{Chen:Moeller:Zagier}    by    D.~Chen,
M.~M\"oller, and D.~Zagier.

%--------------------------------------------------------------------
\subsection{Dynamics on other families of complex varieties.}

One  more challenging direction of study is dynamics of the complex
Hodge   bundle   over   geodesic   flows   over   moduli   spaces  of
higher-dimensional complex manifolds.

\begin{Problem}
Study  dynamics  of  the  Hodge  bundle  over geodesic flows on other
families  of  compact  varieties.  Are  there other dynamical systems
(compared   to   billiards   in   rational   polygons)   which  admit
renormalization leading to dynamics on families of complex varieties?
\end{Problem}

Some  experimental  results for families of Calabi--Yau varieties are
recently   obtained   by  M.~Kontsevich~\cite{Kontsevich:Calabi:Yau}.
S.~Filip studied in~~\cite{Filip:K3} families of K3-surfaces.

%--------------------------------------------------------------------
\subsection*{Acknowledgments}
The  first  author  is  deeply  indebted  to  the  organizers  of the
conference  ``What's Next? The mathematical legacy of Bill Thurston''
for  the extraordinary job which they have done. We thank MPIM, where
part  of  this  paper was written, for extremely stimulating research
atmosphere.

The numerical computations with square-tiled surfaces
have been done with the computer software \cite{sage}.

%--------------------------------------------------------------------

\end{document}